\documentclass[10pt,leqno]{amsart}

\usepackage{graphicx}
\usepackage{amsfonts,delarray,amssymb,amsmath,amsthm,a4,a4wide}
\usepackage{latexsym}
\usepackage{epsfig}
\usepackage{color}
\usepackage[margin=0.76in]{geometry}

\newcommand{\dist}{\text{dist}} %distance
\newcommand{\abs}[1]{\left\vert#1\right\vert}
\newcommand{\diam}{\text{diam}}
\newcommand{\diag}{\text{diag}}
\newcommand{\trace}{\text{trace}}
\newcommand{\loc}{\text{loc}}
\newcommand{\R}{{\mathbb R}} %%reals

\newcommand{\e}{\varepsilon}

\newcommand{\de}{\delta}

\newcommand{\tl}{\tilde\lambda}
\newcommand{\tL}{\tilde\Lambda}
\newcommand{\ld}{\lambda}
\newcommand{\Ld}{\Lambda}

\newcommand{\Sn}{\mathbb S^n}

\newcommand{\eps}{\epsilon}

\newcommand{\p}{\partial}

\newenvironment{myindentpar}[1]%
{\begin{list}{}%
         {\setlength{\leftmargin}{#1}}%
         \item[]%
}
{\end{list}}

\theoremstyle{plain}
\newtheorem{thm}{Theorem}[section]

\newtheorem{lem}[thm]{Lemma}

\newtheorem{prop}[thm]{Proposition}
\newtheorem{rem}[thm]{Remark}

\theoremstyle{definition}
\newtheorem{defn}[thm]{Definition}

\title[ Harnack inequality via sliding paraboloids]{On the Harnack inequality for degenerate and singular elliptic equations with unbounded lower order terms
via sliding paraboloids}
\author{Nam Q. Le}
\address{Department of Mathematics, Indiana University,
Bloomington, 831 E 3rd St, IN 47405, USA.}
\email{nqle@indiana.edu}
\thanks{The research of the author was supported in part by the National Science Foundation under grant DMS-1500400.}
\subjclass[2010]{35B45, 35B65, 35J60, 35J70}
\keywords{Harnack inequality, linearized Monge-Amp\`ere equation, degenerate and singular equations, sliding paraboloids, covering lemma}

\makeatletter
\def\l@subsection{\@tocline{2}{0pt}{2.5pc}{5pc}{}}
\makeatother

\begin{document}
%\date{\today}

\begin{abstract}
We use the method of sliding paraboloids
to establish a Harnack inequality for linear, degenerate and singular elliptic equation with unbounded lower order terms. 
The equations we consider include uniformly elliptic equations and linearized Monge-Amp\`ere equations. Our argument allows us to prove the doubling estimate for functions which, at points
of large gradient, are solutions of (degenerate and singular) elliptic equations with unbounded drift.
\end{abstract}
 \maketitle

\noindent

\section{Introduction and statement of the main result}

This paper is concerned with establishing a Harnack inequality for linear, degenerate and singular elliptic equations with unbounded lower order terms via the method of sliding paraboloids originated
from works of Cabr\'e \cite{Cab} and Savin \cite{Sa}. 
The equations we consider include uniformly elliptic equations as in Krylov-Safonov's theory \cite{KS2} and degenerate and singular elliptic, linearized Monge-Amp\`ere equations as in Caffarelli-Guti\'errez's theory
\cite{CG}. Our argument allows us to prove the doubling estimate for functions which, at points
of large gradient, are solutions of (degenerate and singular) elliptic equations with unbounded drift. These functions have been recently studied by Imbert-Silvestre \cite{IS} and
Mooney \cite{Mooney}.

We begin by recalling the theory of Krylov and Safonov. In 1979, Krylov-Safonov \cite{KS2} established their celebrated Harnack inequality for 
nonnegative solutions to linear, uniformly elliptic equations in non-divergence form
\begin{equation}\sum_{i, j=1}^{n}a^{ij}(x)\frac{\p^2 v}{\p x_i \p x_j}(x)=0
 \label{unieq}
\end{equation}
where the eigenvalues of the measurable, symmetric coefficient matrix $A(x)= \left(a^{ij}(x)\right)_{1\leq i, j\leq n}$ are bounded between two positive constants $\lambda$ and $\Lambda$, that is
\begin{equation}
 \label{unilam}
 \lambda |\xi|^2 \leq \sum_{i, j=1}^na^{ij}(x)\xi_i\xi_j\leq \Lambda |\xi|^2~\text{for all}~x,~\text{and }\xi\in\R^n.
\end{equation}
Krylov-Safonov's Harnack inequality states that if $v\in W^{2, n}_{\loc}(\Omega)$ is a
nonnegative solution of (\ref{unieq}) in a domain $\Omega\subset \R^n$, then
for any ball $B_{2r}(x_0)\subset\subset\Omega$, we have
\begin{equation}
 \label{HI1}
 \sup_{B_r(x_0)} v\leq C(n, \lambda, \Lambda) \inf_{B_r(x_0)} v. 
\end{equation}
There are many proofs of Krylov-Safonov's Harnack inequality. In this paper, we focus on their original proof which has been successfully extended
 to other equations including fully nonlinear equations, degenerate and singular equations, and equations with lower order terms.

 The idea of the proof is the following; see also Caffarelli-Cabr\'e \cite{CC}. In the first step, we show that the distribution function of $v$, $|\{v>t\}|$ decays 
 like $t^{-\varepsilon}$ ($L^{\e}$ estimate) for some small positive constant $\e$ depending only on $n,\lambda,$ and $\Lambda$. 
Thus, up to universal constants depending only on $n,\lambda,$ and $\Lambda$, $v$ is comparable to $v(x_0)$ in $
B_r(x_0)$ except a set of very small measure. 
In the second step, we show that
if $v(x_\ast)$ is much larger than $v(x_0)$ at some point $x_\ast$ in $B_r(x_0)$, then by the same method (now applying to $C_1- C_2 v$), we also find that $v$ is much larger 
than $v(x_0)$ in a subset
of $B_r(x_0)$ of positive measure. This
contradicts the $L^{\e}$ estimate in the first step.

There are three main estimates in the proof of the $L^{\e}$ estimate.
\begin{myindentpar}{1cm}
{\bf Measure estimate:} This estimate says that if $v\in W^{2, n}_{\loc}(\Omega)$ is a
nonnegative supersolution to (\ref{unieq}) in $B_{3r}(x_0)\subset\subset \Omega$ and $v$ is small at least at a point, say, less than or equal to 1, in $B_r(x_0)$ then,
the superlevel set $\{v>M(n,\lambda,\Lambda)\}\cap B_{2r}(x_0)$  in the double ball $B_{2r}(x_0)$ covers at most $(1-\delta (n,\lambda,\Lambda))$ fraction of its measure, for some universal constants $\delta
(n,\lambda,\Lambda)\in (0, 1)$ and $M(n,\lambda,\Lambda)>0$.
This estimate is not difficult and is traditionally based on the Aleksandrov-Bakelman-Pucci (ABP) maximum principle (see, for example, Gilbarg-Trudinger \cite[Theorem 9.1]{GT}). In the ABP estimate, we need the lower bound 
on the determinant of the coefficient matrix which is the case here.\\
{\bf  Doubling estimate:} This estimate says that if $v\in W^{2, n}_{\loc}(\Omega)$ is a
nonnegative supersolution to (\ref{unieq}) in $B_{3r}(x_0)\subset\subset \Omega$ and $v$ is at least 1 in $B_r(x_0)$ then in the double ball $B_{2r}(x_0)$, the value of
$v$ is not smaller than a positive constant $c(n,\lambda,\Lambda)$. This estimate is traditionally based on the construction of subsolutions.\\
{\bf Power decay estimate:} This estimate establishes the geometric decay of $\{v>t\}\cap B_r(x_0)$. It is based on the measure and doubling estimates and a covering lemma which is a consequence
of geometric properties of Euclidean balls. 
\end{myindentpar}
The next two sections will be devoted to reviewing some interesting and fairly new ideas in obtaining the measure and doubling estimates, 
used  in the proofs of Harnack inequality for non-divergence form elliptic equations, other than the ones mentioned above.
\subsection{On measure estimate}
In \cite{Cab}, Cabr\'e devised a new method to prove measure estimate in order to prove a Harnack inequality on Riemannian manifolds with nonnegative sectional curvature. The idea
is to 
estimate the measure of the superlevel set of supersolutions by sliding some specific functions from below and estimating the measure of 
the set of contact points. In \cite{Cab}, Cabr\'e uses the distance function squared which is a natural replacement of quadratic 
polynomials in a Riemannian manifold. 

In \cite{Sa}, Savin introduced the method of sliding paraboloids (see Definition \ref{parabo_def}) from below to obtain the measure estimate thus bypassing the use of ABP estimate. Here paraboloids of constant opening are slided
from below till they touch the graph of $v$ for the first time. These are the points where we use the equation and obtain the lower bound for the measure of the 
touching points.  Since the values of the solution at the touching points are universally bounded from above, we obtain the measure estimate. This method was later extended by Wang \cite{W} to parabolic equations. 

There has been an increasing number of papers inspired by Savin's method. We would like to mention three recent papers \cite{AS, IS, Mooney} that are most relevant to the content of our paper.

In \cite{AS}, Armstrong and Smart study regularity and stochastic homogenization of fully nonlinear equations without uniform ellipticity. In order to carry out their analysis,
they prove new deterministic regularity estimates in which the dependence on a uniform upper bound for the ellipticity is replaced by that of its $L^n$-norm.
One of the key arguments in \cite{AS} involves touching from below by translation of the singular function $|x|^{-\alpha}$, for suitably large $\alpha$, which
is a subsolution of the equation under consideration. 

In \cite{IS}, Imbert and Silvestre establish a
Harnack inequality for functions which, at points of large gradient, are solutions to linear, uniformly elliptic equations
\begin{equation}a^{ij} v_{ij} + b\cdot Dv=f
 \label{ISeq}
\end{equation}
where the drift $b$ is bounded. 
The measure estimate in \cite{IS} was established by touching the graph of $v$ from below by translation of the cusps functions of the form $-|x|^{1/2}$.
In \cite{Mooney}, Mooney establishes the
Harnack inequality for these equations
with $L^n$ drift $b$. One main contribution of \cite{Mooney} is a new proof of the measure estimate in \cite{IS} that uses sliding of paraboloids from below at all scales
and a set decomposition algorithm.
\subsection{On doubling estimate}
An attractive feature in Mooney's proof of the doubling estimate for (\ref{ISeq}) in \cite{Mooney} is to slide barriers $\varphi(x)$, of the form $|x|^{-\alpha}$ for suitably 
large $\alpha$,  from below until they touch the graph of $v$ and 
use the equation at the contact points. The barriers are indeed subsolutions of (\ref{ISeq})
when the drift $b$ and the nonhomogeneous term $f$ are bounded. One subtle point in this method is the translation invariant property of barriers. Roughly speaking, this says that if
$$a^{ij}(x) \varphi_{ij}(x)\geq 1
~\text{in}~ B_4(0)\setminus B_1 (0)$$
then we also have $$a^{ij}(x)\varphi_{ij}(x-y)\geq 1~\text{in}~B_3(0)\setminus B_2 (0) ~\text{for } y\in B_1(0).$$
This property holds for uniformly elliptic equations like (\ref{ISeq}) with $(a^{ij})$ satisfying (\ref{unilam}) while it does not hold for degenerate elliptic equations such as the linearized Monge-Amp\`ere equations; see (\ref{LMAeq})
in the next section.
\subsection{Doubling estimate via sliding paraboloids}
Motivated by Mooney's approach to the doubling estimate and Savin's sliding paraboloid method in obtaining the measure estimates, we slide generalized 
paraboloids (see Definition \ref{gparabo_def}) from below to obtain the doubling estimate for 
linear, degenerate and singular elliptic equations that include uniformly elliptic equations as in Krylov-Safonov's theory \cite{KS2} and 
linearized Monge-Amp\`ere equations as in Caffarelli-Guti\'errez's theory \cite{CG}.
This in turn also works for equations with unbounded lower order terms. The integrability assumptions on the lower terms are similar to those of uniformly elliptic equations
studied in the work of
Trudinger \cite{Tr}.

The key observation is simple: if $v$ is a nonnegative supersolution of an elliptic equation
$a^{ij} v_{ij}\leq 0$
then $-v^{-\e}$ is also a supersolution. This property does not depend on the ellipticity of the coefficient matrix $(a_{ij})$. Thus, we can slide 
(generalized) paraboloids from below to touch the graph of
$-v^{-\e}$. However, due to the degeneracy and singularity nature of the equations considered in this paper, in order to obtain the doubling estimate, we need to modify the paraboloids by certain convex function at points where there is a lack of
uniform ellipticity. This modification is also needed in the traditional proof using subsolution construction; see the function $h_\e$ and $h_\delta$ in
the proofs of Lemmas \ref{double_lem1} and \ref{double_lem}.

Our argument also applies to functions which are solutions to equations of linearized Monge-Amp\`ere type, but only at points where the gradient is large; 
see Remark \ref{largeD_rem} and Lemma \ref{large_lem}.

\subsection{Harnack inequality via sliding paraboloids}
Combining our new argument in establishing the doubling estimate with Savin's measure estimates, we can prove a Harnack inequality for
linear, degenerate and singular elliptic equations with unbounded lower order terms via the method of sliding paraboloids. We hope that this unified treatment of the estimates in proving Harnack
inequality will have further applications in degenerate equations.
\subsection{Statement of the main result}
We now state precisely the main result proved in this paper. Let $\lambda,\Lambda,\tl$ and $\tL$ be given positive constants with $\lambda\leq \Lambda$ and $\tl\leq \tL$.
Let $\Omega\subset\R^n$ be an open, convex and bounded domain. Assume that a strictly convex function $u$ satisfies the Monge-Amp\`ere equation
\begin{equation}
\lambda\leq \det D^{2} u\leq \Lambda~\text{in}~\Omega
\label{pinch1}
\end{equation}
in the sense of Aleksandrov; see, for example, Guti\'errez \cite[Definition 1.2.1]{G}. Then $u\in C^{1,\alpha_\ast}$ for some $\alpha_\ast$ depending only on $n,\lambda,\Lambda$; see, for example \cite[Theorem 5.4.5]{G}
and Theorem \ref{C1alpha} in Section \ref{Prelim_sec}. 
We additionally assume throughout the paper, as in \cite{CG, Mal}, that $u\in C^2(\Omega)$. This is because crucial steps in proving 
our Harnack inequality require the second-order differentiations in $u$; see, for instance, the steps leading to (\ref{uvD22}) and (\ref{Dxy}). The estimates in the paper, however, do not depend on the $C^2$ norm of $u$ nor
any modulus of continuity of $\det D^2 u$. Quantitative
estimates involving $u$ depend only on $n$, $\lambda$ and $\Lambda$ because for these, we use the results in Sections 
\ref{MA_sec}--\ref{ink_sec} which hold for $u$ satisfying (\ref{pinch1}) only.
Throughout, we denote the cofactor matrix of the Hessian matrix $D^2 u= \left(u_{ij}\right)_{1\leq i, j\leq n}$ by 
$$U = (\det D^2 u)(D^2 u)^{-1}\equiv (U^{ij})_{1\leq i, j\leq n}.$$
Let $A= (a^{ij})_{1\leq i, j\leq n}$  be a symmetric matrix satisfying
\begin{equation}\tilde\lambda U\leq A\leq\tilde\Lambda U.
 \label{unilamU}
\end{equation}
We will denote by $S_u(x, h)$ the section of $u$ centered at $x\in\Omega$ with height $h$; see Definition \ref{sec_def}.

Our main result is the following Harnack inequality for the linear, degenerate and singular elliptic equations, which are $(D^2 u)^{-1}$-like, with unbounded lower order terms. By abuse of terminology,
we call these equations
linearized Monge-Amp\`ere equations associated with the potential $u$.
\begin{thm} [Harnack inequality]
\label{CGthm} Assume that (\ref{pinch1}) and (\ref{unilamU}) are satisfied in $\Omega$. 
Suppose that $v\geq 0$ is a $W^{2, n}_{\loc}(\Omega)$ solution of 
\begin{equation}a^{ij}v_{ij} + b\cdot Dv + cv=f
 \label{LMAeq}
\end{equation}
in a section $S:=S_u (x_0, 2h)\subset\subset \Omega$ where 
$f\in L^n (S), c\in L^n(S)$ and $b\in L^p(S)$ 
with $p>\frac{n(1+\alpha_\ast)}{2\alpha_\ast}$. 
Then
\begin{equation}\sup_{S_{u}(x_0, h)} v\leq C(n, \lambda, \Lambda,\tl,\tL)^{N(h, h_0)} \left(\inf_{S_{u}(x_0, h)} v+ h^{1/2}\|f\|_{L^n(S)}\right).
 \label{HI2}
\end{equation}
The function $N(h, h_0)$ and $h_0$ are defined as follows:
\begin{equation*}
 N(h, h_0)=\max\left\{1, \left(\frac{h}{h_0}\right)^{n/2}\right\}
\end{equation*}
where, for some universal constant $\e_5(n,p, \lambda,\Lambda,\tl,\tL)$,
$$h_0^{\frac{\alpha_\ast}{1+\alpha_\ast}-\frac{n}{2p}}\|b\|_{L^p(S)}\leq \e_5,~\text{and}~h_0^{1/2} \|c\|_{ L^n (S)}\leq \e_5.$$
\end{thm}
\begin{rem}
 We note that, under the assumption (\ref{pinch1}), 
$D^2 u$ is not bounded and thus (\ref{LMAeq}), though elliptic, can be both degenerate and singular. In fact, the best regularity we can have from (\ref{pinch1}) only is $D^2 u\in L^{1+\varepsilon}(\Omega)$ where 
$\varepsilon$ is a small constant depending on $n,\lambda$ and $\Lambda$. This follows from De Philippis-Figalli-Savin's and Schmidt's interior $W^{2, 1+\varepsilon}$ estimates 
\cite{DPFS, Sch} for the Monge-Amp\`ere equation. 
\end{rem}
\begin{rem}~
\begin{myindentpar}{0.8cm}
(i) We will sometimes use (\ref{LMAeq}) in the form
$\trace (A D^2 v) + b\cdot Dv + cv=f.$\\
(ii) When $u=\frac{1}{2}|x|^2$, $U= I_n$ and (\ref{LMAeq}) becomes a linear, uniformly elliptic equations in non-divergence form. In this case, we can take $\alpha^{\ast}=1$.\\
(iii) When $\tl=\tL=1$, $A=U$ and 
 (\ref{LMAeq}) becomes a linearized Monge-Amp\`ere equation with lower order terms.
\end{myindentpar}
\end{rem}
\begin{rem}
 When $b=0$ and $c=0$, we can take $h_0=\infty$ and $N(h, h_0)=1$ in Theorem \ref{CGthm}. 
 In this case, Theorem \ref{CGthm} was discovered in the fundamental work of Caffarelli-Guti\'errez \cite[Theorem 5]{CG} for $f=0$ with the extension to the nonhomogeneous case by Trudinger-Wang
 \cite[Theorem 3.4]{TW3}.
\end{rem}

Theorem \ref{CGthm} is also related to a recent paper by Maldonado \cite{Mal}. In \cite{Mal}, Maldonado 
established an interior Harnack inequality for nonnegative solutions to the nonhomogeneous linearized Monge-Amp\`ere equation with lower order terms where the drift term 
has a special structure and the coefficient $c$ has a sign. Certain boundedness on the lower order term was assumed. More precisely, Maldonado considered equations of the type
\begin{equation}\trace \left((D^2 u)^{-1}D^2 v\right) + b\cdot \left((D^2 u)^{-1/2} Dv\right) + cv = f.
 \label{DCLMA}
\end{equation}
Here $u$ is a strictly convex, twice continuously differentiable function on $\R^n$ whose associated Monge-Amp\`ere measure $\mu_{u}(x):=\det D^2 u(x)dx$ is stricly positive and satisfies
a doubling property; the functions  $b, c, f$ satisfy
$$b, c, f\in L^{\infty};~c\leq 0; ~\text{and } (D^2 u)^{-1/2}b\in L^n.$$
These conditions on the lower order terms allow the use of the
ABP maximum principle to (\ref{DCLMA}).
The Harnack inequality in \cite{Mal} extends the fundamental work of Caffarelli-Guti\'errez \cite{CG} in the case with no lower order terms and the 
Monge-Amp\`ere measure $\mu_{u}(x)$ satisfies  the $(\mu_{\infty})$-condition which is stronger than the doubling condition. 

The proof of Harnack inequality in \cite{Mal} was
proved by means of real analysis techniques in 
spaces of homogeneous type. These include exploiting the variational structure of (\ref{DCLMA}), local-BMO estimates and Poincar\'e-type inequality within the Monge-Amp\`ere 
quasi-metric structure. This proof
differs from that in \cite{CG} and in this paper. Here, we also slide paraboloids and then using normalization techniques as in \cite{CG}.
\begin{rem}
We require the high integrability of the drift term $b$ in Theorem \ref{CGthm} in order to obtain a small $L^n$-norm of the rescaled drift when we rescale (\ref{LMAeq}); see
Section \ref{AIP_sec} and Lemma \ref{res_lem}. This will allow us to apply the
 measure estimate in  Lemma \ref{meas_lem} and the doubling estimate in Lemma \ref{double_lem} to the rescaled equation.
\end{rem}

\begin{rem}
 We expect that the measure localization technique of Safonov \cite{Safo} can be combined with the sliding paraboloids argument in this paper to prove the Harnack inequality
 for solutions to (\ref{LMAeq}) when $b\in L^{\frac{n(1+\alpha_*)}{2\alpha_*}}.$ It seems that a new method is required to establish the Harnack inequality for 
 (\ref{LMAeq}) when $b\in L^n$.
\end{rem}

In this paper, a positive constant depending only on $p, \lambda,\Lambda,\tl,\tL$ and the dimension $n$  of $\R^n$ is called {\it universal}. We denote universal constants by $c_1, C, C_1, C_2, K, M,\delta, \cdots,$ etc, 
and their values may change from line to
line. We use $C(\cdot, \cdot)$ to emphasize the dependence of $C$ on the parameters in the parentheses. Universal constants in Sections \ref{measure_sec}
and \ref{double_sec} do not depend on $p$. Repeated indices are summed, such as $a^{ij} v_{ij}=\sum_{i, j=1}^n a^{ij} v_{ij}.$

We can assume that all functions $u$, $v$ in this paper are smooth and thus solutions can be interpreted in the classical sense. However, our 
estimates do not depend on the assumed smoothness but only on the given structural constants. Thus, by a standard approximation argument (see, for example \cite[Section 9.1]{GT}),
our estimates also hold for $v\in W^{2, n}_{\loc}.$

The rest of the paper is organized as follows. In Section \ref{Prelim_sec}, we present some notation and background results needed in the rest of the paper.
The following sections of the paper use these results but otherwise are self-contained and quite elementary.
We prove the measure estimate, Lemma \ref{meas_lem}, in Section \ref{measure_sec}. In Section \ref{double_sec}, we prove the doubling estimate, Lemma \ref{double_lem}.
In Section \ref{decay_sec}, we prove the power decay estimate, Theorem \ref{decay_thm}. Theorem \ref{CGthm} will be proved in
the final Section \ref{Harnack_sec}. 

\section{Preliminaries}
\label{Prelim_sec}
In this section, we present some notation and background results needed in the rest of the paper. These 
include generalized paraboloid associated with a convex function, section of a convex function, the Aleksandrov maximum principle, John's lemma, basic linear algebra facts, key geometric properties of solutions to the Monge-Amp\`ere equation, the Vitali covering lemma and the growing ink-spots lemma in the context of 
the Monge-Amp\`ere equation.
\subsection{Notation}
We begin with notation used throughout the paper.\\
\begin{tabular}{lll}
$\quad$ $|x|:$ && the Euclidean length of $x=(x_1,\cdots, x_n)\in\R^n$, $\displaystyle |x|=\left(\sum_{i=1}^n x_i^2\right)^{1/2},$\\
$\quad$ $x\cdot y:$ && the dot product of $x=(x_1,\cdots, x_n), y= (y_1,\cdots, y_n) \in\R^n$, $\displaystyle x\cdot y= \sum_{i=1}^n x_i y_i,$\\
$\quad$ $B_r(y):$ &&the open ball centered at $y\in\R^n$ with radius $r>0,$\\
$\quad$ $u^{\pm}$: && $u^{\pm}=\max (\pm u, 0)$ for $u\in\R$,\\
$\quad$ $\overline E:$ && the closure of a set $E,$\\  
$\quad$ $\p E:$ && the boundary of a set $E,$\\  
$\quad$ $|E|:$ && the Lebesgue measure of a set $E,$\\  
$\quad$ $\diam (E):$ && the diameter of a bounded set $E,$\\ 
$\quad$ $\dist (\cdot, E):$ && the distance function from a closed set $E,$\\ 
$\quad$ $Du:$ && the gradient of a function $u:\R^n\rightarrow \R$, $Du= (u_1,\cdots, u_n)\equiv \left(\frac{\partial u}{\partial x_i}\right)_{1\leq i\leq n},$\\
$\quad$ $D^2u:$ && the Hessian matrix of $u:\R^n\rightarrow \R$, $D^2u= \left(u_{ij}\right)_{1\leq i, j\leq n}\equiv \left(\frac{\partial^2 u}{\partial x_i \partial x_j}\right)_{1\leq i, j\leq n},$\\
$\quad$ $\Sn:$ && the space of symmetric $n\times n$ matrices,\\
$\quad$ $A, B\in\Sn$, $A\geq B$: && if the eigenvalues of $A-B$ are nonnegative,\\
$\quad$ $\trace (M):$ && the trace of a matrix  $M$;\\
$\quad$ $\|M\|:$ && the Hilbert-Schmidt norm of $M\in \Sn$: $\|M\|^2=\trace(M^T M)$.\\
\end{tabular}

\subsection{Basic definitions}
\begin{defn}[Paraboloid]
\label{parabo_def}
 A paraboloid $P$ of opening $a$ and vertex $y\in \R^n$ is a quadratic function of the form
 $$P(x)= C-\frac{a}{2}|x-y|^2$$
 for some constant $C\in \R$.
\end{defn}
\begin{defn}[Generalized paraboloid]
\label{gparabo_def}
 A generalized paraboloid $P$, associated with a convex, $C^1$ function $u$, of opening $a$ and vertex $y\in \R^n$ is a function of the form
 $$P(x)= C-a[u(x)- u(y)-Du(y)\cdot (x-y)]$$
 for some constant $C\in \R$.
\end{defn} 
Clearly, Definition \ref{parabo_def} is a special case of Definition \ref{gparabo_def} with $u(x) =\frac{1}{2}|x|^2$.

We will frequently use the following lemma, due to Fritz John \cite[Theorem 3]{John}.
\begin{lem}[John's lemma]\label{John_lem}
Let $\Omega\subset \R^{n}$ be a compact, convex set with nonempty interior. Then there is an affine transformation $Tx = Ax + b$ where the $n\times n$ matrix $A$ is invertible and $b\in \R^n$ such that
$$B_1 (0)\subset T(\Omega)\subset B_n(0).$$
\end{lem}
In Lemma \ref{John_lem}, the transformation $T$ is said to {\it normalize} $\Omega$. 
\begin{defn} [Normalized convex set]
 We say that a convex set $\Omega\subset\R^n$ is normalized if 
 $B_1(0)\subset\Omega\subset B_n(0).$
\end{defn}
A central notion in the theory of Monge-Amp\`ere equation is that of sections, introduced and investigated by 
Caffarelli; see, for example \cite{C2}. They will play the role that balls do in the uniformly elliptic equations.
\begin{defn}[Section]
\label{sec_def}
Let $u\in C^{1}(\Omega)\cap C(\overline{\Omega})$ be a convex function in $\Omega$. The section of $u$ centered at $x$ with height $h$, denoted by
$S_{u}(x, h)$, is defined by
$$S_u(x, h) =\{y\in\overline{\Omega}: u(y) < u(x) + D u(x)\cdot (y-x) + h\}.$$
\end{defn}
If $u(x)=\frac{1}{2}|x|^2$ in $\overline{\Omega}$ then $S_u(x, h)= B_{\sqrt{2h}}(x)\cap \overline{\Omega}$.
\subsection{Some useful estimates}
We will use the following gradient estimate for a convex function.
\begin{lem} [Gradient estimate]
\label{slope-est}
Let $\Omega\subset \R^n$ be a bounded convex set and $u\in C^1(\Omega)\cap C(\overline{\Omega})$ a convex function in $\overline{\Omega}$. If $x\in\Omega$, then
$$|Du(x)|\leq \frac{\max_{y\in\p\Omega} u (y) -u(x)}{\dist (x, \p\Omega)}.$$
\end{lem}
\begin{proof}[Proof of Lemma \ref{slope-est}] We include here a simple proof for reader's convenience.
Let $r:=\dist (x,\p\Omega)$. Then, for each $\e>0$, 
$y_0 = x + r \frac{Du(x)}{|Du(x)|+\e}\in\Omega$. Hence, by convexity
$$\max_{y\in\p\Omega} u(y)\geq u(y_0) \geq u(x) + Du(x)\cdot (y_0-x)=u(x) + r \frac{|Du(x)|^2}{|Du(x)|+\e}.$$
Thus
$$|Du(x)|-\e\leq \frac{|Du(x)|^2}{|Du(x)|+\e}\leq \frac{\max_{y\in\p\Omega} u (y) -u(x)}{\dist (x,\p\Omega)}$$
and by letting $\e\rightarrow 0$, we obtain the desired estimate.
\end{proof}
We recall the Aleksandrov maximum principle; see for example \cite[Theorem 1.4.2]{G}.
\begin{thm}[Aleksandrov maximum principle]
Let $\Omega\subset \R^n$ be a bounded, open and convex set, and $u\in C(\overline{\Omega})$ a convex function with $u=0$ on $\p\Omega$. Then, for all $x_0\in\Omega$, we have
$$|u(x_0)|^{n}\leq C(n) [\diam (\Omega)]^{n-1} \dist (x_0, \p\Omega)\int_\Omega \det D^2 u(x) dx.$$
\label{Alekmp}
\end{thm}
We will use the following linear algebra lemma.
\begin{lem} [Matrix inequalities]
 \label{LA_lem}
  Assume that $A, B\in \Sn$.
  \begin{myindentpar}{1cm}
  (a) If $A, B\geq 0$ then
  $\displaystyle\trace(AB) \geq n(\det A \det B)^{\frac{1}{n}}.$\\
  (b) Suppose that $B$ is positive definite. If $A\geq -aB$ for some $a\geq 0$ and $\trace (B^{-1}A)\leq D$ then $$(an + D)B\geq A\geq -aB.$$
  (c) Suppose that $A=(a_{ij})_{1\leq i, j\leq n}$ is positive definite. Then, for any $b=(b_1,\cdots, b_n)\in \R^n$,
 we have
 $$a_{ij} b_i b_j\geq \frac{|b|^2}{\trace(A^{-1})}.$$
  \end{myindentpar}
\end{lem}
\begin{proof}[Proof of Lemma \ref{LA_lem}] The proof of (a) is standard so we omit it.\\
 (b) Let $I_n\in \Sn$ be the identity matrix. Since $B$ is positive definite, we can rewrite the hypotheses as $$B^{-1/2} A B^{-1/2}\geq -aI_n,~\text{and } \trace(B^{-1/2} A B^{-1/2})\leq D.$$
Hence $$B^{-1/2} A B^{-1/2}\leq  [a(n-1) + D]I_n\leq (an+ D) I_n$$
and therefore, we obtain $A\leq (an+ D) B$ as asserted.\\
(c) Let $P$ be an orthogonal matrix such that $A= P D P^t$ where $D$ is the diagonal matrix: $$D=\diag (\lambda_1,\cdots,\lambda_n).$$ Then
 $$\trace(A^{-1})=\trace (D^{-1})~\text{and}~a_{ij}b_i b_j= (Ab)\cdot b= (P D P^t b)\cdot b= (D P^t b)\cdot (P^t b).$$
 Because $|P^t b|= |b|$, it suffices to prove the lemma for the case $A=D=\diag (\lambda_1,\cdots,\lambda_n).$ In this case, the lemma is equivalent to proving the obvious inequality:
 $$\left(\sum_{i=1}^n\lambda_i^{-1}\right)\left(\sum_{i=1}^n\lambda_i b_i^2\right)\geq \sum_{i=1}^n b_i^2.$$
\end{proof}
\subsection{Review on the Monge-Amp\`ere equation}
\label{MA_sec} 
Most of the results in this section are taken from Guti\'errez \cite{G}.
We assume throughout this Section \ref{MA_sec} that
$u$ is a strictly convex solution to the Monge-Amp\`ere equation
$$\lambda \leq \det D^2 u \leq \Lambda~\text{in an open, bounded and convex domain}~\Omega\subset\R^n.$$
We will use the following volume growth for compactly supported sections; see \cite[Corollary 3.2.4]{G}.
\begin{lem}[Volume estimate for section]\label{vol-sec1}
Suppose that $S_u(x, h)\subset\subset \Omega.$
Then
$$[C(n, \lambda, \Lambda)]^{-1}h^{n/2} \leq |S_u(x, h)| \leq C(n, \lambda, \Lambda) h^{n/2}.$$
\end{lem}
We will use the engulfing property of sections; see \cite[Theorem 3.3.7]{G}.
\begin{thm} [Engulfing property of sections] There is a universal constant $\theta_0(n,\lambda,\Lambda)>2$ with the following property: If $S_u(y, 2h)\subset\subset\Omega$
and 
$x\in S_u(y, h)$, then we have $S_{u}(y, h)\subset S_u(x, \theta_0 h).$
\label{engulfthm}
\end{thm}
We will use the following estimate on the size of sections; see \cite[Theorem 3.3.8]{G}.
\begin{thm} [Size of section]\label{sec-size}
Suppose that $S_u(0, t_0)\subset\subset\Omega$ is a normalized section, that is,
$B_{1}(0)\subset S_u(0, t_0)\subset B_n(0).$
Then there are universal constants $\mu(n,\lambda,\Lambda)\in (0,1)$ and $C(n,\lambda,\Lambda)$ such that for all sections $S_u(x, h)\subset\subset\Omega$ with $x\in S_u(0, 3/4 t_0)$, we have
$$S_u(x, h)\subset B_{Ch^{\mu}}(x).$$
\end{thm}
We will use the following inclusion and exclusion property of sections; see \cite[Theorem 3.3.10]{G}.
\begin{thm} [Inclusion and exclusion property of sections]
 \label{pst}
There exist universal constants $c_0(n,\lambda,\lambda)>0$ and $p_1(n,\lambda,\Lambda)\geq 1$ such that 
 \begin{myindentpar}{1cm}
 (i) if $0< r<s\leq 1$ and $x_1\in S_u(x_0, r t)$ where $S_u(x_0, 2t)\subset\subset \Omega$, then
 $$S_{u}(x_1, c_0 (s-r)^{p_1} t)\subset S_{u}(x_0, s t);$$
 (ii)  if $0< r<s< 1$ and $x_1\in S_u(x_0, t)\backslash S_{u}(x_0, st)$ where $S_u(x_0, 2t)\subset\subset
 \Omega$, then
 $$S_{u}(x_1, c_0 (s-r)^{p_1} t)\cap S_{u}(x_0, rt)=\emptyset.$$
 \end{myindentpar}
In fact, we can choose $p_1= (n+1)\mu^{-1}$ where $\mu$ is the constant in Theorem \ref{sec-size}.
\end{thm}
We recall the $C^{1,\alpha_\ast}$ regularity property of strictly convex solution to the Monge-Amp\`ere equation; see \cite[Theorem 5.4.5]{G}.
\begin{thm}
[$C^{1,\alpha_\ast}$ regularity]
\label{C1alpha} There exists a universal constant
$\alpha_\ast(n,\lambda,\Lambda)\in (0, 1]$ such that $u$ is $C^{1,\alpha_\ast}(\Omega)$. More quantitatively, if $S_u(x_0, t)\subset\subset\Omega$
is a normalized section and $x, y\in S_u(x_0, t/2)$ then $$|Du(x)-Du(y)|\leq C(n,\lambda,\Lambda)|x-y|^{\alpha_\ast}.$$
\end{thm}
\begin{lem}
 \label{KKlem}
 There exist universal constants $K(n,\lambda,\Lambda),\hat K(n,\lambda,\Lambda)>1$ with the following properties:
\begin{myindentpar}{1cm}
 (i) Let $\theta_0$ be as in Theorem \ref{engulfthm}. If  $$S_u(x_1, 4\theta_0 h_1), S_u(x_2, 4\theta_0 h_2)\subset\subset\Omega, ~S_u(x_1, h_1)\cap S_u(x_2, h_2)\neq\emptyset~ \text{and}~ 2h_1\geq h_2$$ then $$S_u(x_2, h_2)\subset
S_u(x_1, Kh_1).$$
(ii) If $S_u(x, t)\subset S_u(y, h)$ where $S_u(y, \hat K h)\subset\subset \Omega$ then $S_u(x, Kt)\subset S_u(y, \hat K h).$ 
\end{myindentpar}
\end{lem}
\begin{proof}[Proof of Lemma \ref{KKlem}]
(i) We can choose $K=2\theta_0^2$. Indeed, suppose that
$x\in S_u(x_1, h_1)\cap S_u(x_2,  h_2)$ and $2h_1\geq h_2$ where 
$S_u(x_1, 4\theta_0 h_1), S_u(x_2, 4\theta_0 h_2)\subset\subset\Omega.$ Then, by the engulfing property of sections in Theorem \ref{engulfthm}, $ S_u(x_2, h_2)\subset S_u(x, \theta_0 h_2)
\subset S_u(x, 2\theta_0 h_1)$ and 
$x_1\in S_u(x_1, h_1)
\subset S_u(x, 2\theta_0 h_1)$. Again, by the engulfing property, we have $S_u(x, 2\theta_0 h_1)\subset S_u(x_1, 2\theta^2_0 h_1)$. 
It follows that $S_u(x_2, h_2)\subset S_u(x_1, 2\theta_0^2 h_1)\equiv S_u(x_1, Kh_1)$.\\
(ii) The existence of $\hat K$ is due to Lemma \ref{vol-sec1}
 and Theorem \ref{pst}. Indeed, by Lemma \ref{vol-sec1}, the inclusion $S_u(x, t)\subset S_u(y, h)\subset\subset \Omega$ implies that $t\leq C(n,\lambda,\Lambda) h$. From
 $x\in S_u(y, \hat K h)\subset\subset \Omega$ and Theorem \ref{pst}, we can find $c_1(n,\lambda,\Lambda)$ small, universally such that $S_u(x, c_1 \hat K h)\subset S_u(y, \hat K h).$ Thus
 it suffices to choose $\hat K= K C c_1^{-1}$ to get $S_u(x, Kt)\subset S_u(y, \hat K h)$.
\end{proof}

\subsection{Vitali covering lemma}
In the context of the Monge-Amp\`ere equation, we have the following Vitali covering lemma; see also \cite[Lemma 1]{DPFS}.
\begin{lem}[Vitali covering]\label{Vita_cov} 
Suppose that $u$ is a strictly convex solution to the Monge-Amp\`ere equation
$\lambda\leq \det D^2 u\leq\Lambda$ in a bounded and convex set $\Omega\subset\R^n$. Let $\theta_0$ and $K$ be as in Theorem \ref{engulfthm} and Lemma \ref{KKlem}, respectively. 
\begin{myindentpar}{1cm}
 (i) Let $\mathcal{S}$ be a collection of sections $S^x=S_u(x, h(x))$ where $S_u(x, 4\theta_0 h(x))\subset\subset\Omega$. Then there exists a countable subcollection of disjoint sections 
 $\displaystyle\bigcup_{i=1}^\infty S_u(x_i, h(x_i))$ such that
 $$\displaystyle \bigcup_{S^x\in \mathcal{S}} S^x\subset \bigcup_{i=1}^\infty S_u(x_i, Kh(x_i)).$$
 (ii Let $D$ be a compact set in $\Omega$ and assume that for each $x\in D$ we associate a corresponding section $S_u(x, h(x))\subset\subset\Omega$. Then we can find a 
 finite number of these sections $S_u(x_i, h(x_i)), i=1,\cdots, m,$ such that
$$D \subset \bigcup_{i=1}^m S_u(x_i, h(x_i)),~\text{with}~ S_u(x_i, K^{-1} h(x_i))~\text{disjoint}.$$
\end{myindentpar}
\end{lem}
For reader's convenience, we include its standard proof using
the engulfing property of sections of solutions to the Monge-Amp\`ere equation.
\begin{proof}[Proof of Lemma \ref{Vita_cov}] 
(i) From the volume estimate for sections in Lemma \ref{vol-sec1} and $S_u(x, h(x))\subset\subset \Omega$, we find that
$$H:= \sup\{h(x)| S^x\in\mathcal{S}\}\leq C(n,\lambda, \Lambda,\Omega)<\infty.$$
Define $$\mathcal{S}_i:= \{S^x\in \mathcal{S}| \frac{H}{2^i}<h(x) \leq \frac{H}{2^{i-1}}\}~(i=1,2,\cdots).$$ We define $\mathcal{F}_i\subset \mathcal{S}_i$ as follows. 
Let $\mathcal{F}_1$ be any maximal disjoint collection of sections in $\mathcal{S}_1$. By the volume estimate in Lemma \ref{vol-sec1}, $\mathcal{F}_1$ is finite. 
Assuming $\mathcal{F}_1,\cdots, \mathcal{F}_{k-1} $ have been selected, we choose $\mathcal{F}_k$
to be any maximal disjoint subcollection of
$$\left\{S\in \mathcal{S}_k| S\cap S^{x}=\emptyset~\text{for all~} S^x\in \bigcup_{j=1}^{k-1}\mathcal{F}_j\right\}.$$ Each $\mathcal{F}_k$ is again a finite set. 

We claim that the 
countable subcollection of disjoint sections $S_u(x_i, h(x_i))$ where $\displaystyle S^{x_i}\in \mathcal{F}:=\bigcup_{k=1}^{\infty} \mathcal{F}_k$ satisfies the conclusion of the lemma. To see this, it 
suffices to show that for any section $S^x\in \mathcal{S}$, there exists a section $ S^y\in \mathcal{F}$ such that $S^x\cap S^y\neq \emptyset$ and $S^x\subset S_u(y, Kh(y))$. 
The proof of this fact is simple. There is an index $j$ such that $S^x\subset \mathcal{S}_j$. By the maximality of $\mathcal{F}_j$, there is a section $\displaystyle S^y\in \bigcup_{k=1}^j 
\mathcal{F}_k$ with $S^x\cap S^y \neq\emptyset$. Because $h(y)> \frac{H}{2^j} $ and $h(x) \leq \frac{H}{2^{j-1}}$, we have $h(x) \leq 2 h(y)$.
By Lemma \ref{KKlem}, we have $S^x\subset S_u(y, K h(y))$.
\vglue 0.2cm
\noindent
(ii) We apply (i) to the collection of sections $S_u(x, K^{-1}h(x))$ where $x\in D$. 
Then there exists a countable subcollection of disjoint sections 
 $\displaystyle\bigcup_{i=1}^\infty S_u(x_i, K^{-1}h(x_i))$ such that
 $$\displaystyle D \subset\bigcup_{x\in D}S_u(x, K^{-1}h(x))\subset \bigcup_{i=1}^\infty S_u(x_i, h(x_i)).$$
By the compactness of $D$, we can choose a finite number of sections $S_u(x_i, h(x_i))$ $(i=1,\cdots, m)$ which cover $D$.
\end{proof}
\subsection{Growing ink-spots lemma}
\label{ink_sec}
In the proof of the power decay estimate in Theorem \ref{decay_thm}, we use the following consequence of Vitali's covering
lemma. It is often referred to as the growing ink-spots lemma which was first introduced by Krylov-Safonov \cite{KS2}; see also 
Imbert-Silvestre \cite[Lemma 2.1]{IS}. The term \lq\lq growing ink-spots lemma\rq\rq was coined
by E. M. Landis.
\begin{lem}[Growing ink-spots lemma] \label{inkspots} Suppose that $u$ is a strictly convex solution to the Monge-Amp\`ere equation
$\lambda\leq \det D^2 u\leq\Lambda$ in a bounded and convex set $\Omega\subset\R^n$. Let $\hat K$ be as in Lemma \ref{KKlem}. Assume that for some $h>0$, 
$S_u(0, \hat K h)\subset\subset\Omega.$

Let $E \subset F \subset S_u(0, h)$ be two open sets. Assume that for some constant $\delta \in (0,1)$, the
following two assumptions are satisfied:
\begin{myindentpar}{1cm}
(i) If any section $S_u(x, t) \subset S_u(0, h)$ satisfies $|S_u(x, t) \cap E| > (1-\delta)
  |S_u(x, t)|$, then $S_u(x, t) \subset F$;\\
(ii) $|E| \leq (1-\delta) |S_u(0, h)|$. 
\end{myindentpar}
Then $|E| \leq (1-c_2\delta) |F|$ for some constant $c_2$ depending only on $n,\lambda$ and $\Lambda$.
\end{lem}

\begin{proof} [Proof of Lemma \ref{inkspots}] The proof here follows the line of argument in Imbert-Silvestre \cite{IS} in the case $u(x)=\frac{1}{2}|x|^2$.
For every $x \in F$, since $F$ is open, there exists some maximal section which is contained in $F$ and contains $x$. 
We choose one of those sections for each $x \in F$ and call it $S_u(x,\bar h(x))$.

If $S_u(x,\bar h(x)) = S_u(0, h)$ for any $x \in F$, then the conclusion of the lemma follows immediately since 
$|E| \leq (1-\delta) |S_u(0, h)|$ by (ii) and $F= S_u(0, h)$ in this case, so let us assume that this is not the case here.

We claim that $|S_u(x,\bar h(x)) \cap E| \leq (1-\delta)|S_u(x,\bar h(x))|$. 
Otherwise, we could find a slightly larger section $\tilde S$ containing $S_u(x,\bar h(x))$ such that $|\tilde S \cap E| > (1-\delta) |\tilde S|$ and 
$\tilde S \not\subset F$, contradicting (i).

The family of sections $S_u(x,\bar h(x))$ covers the set $F$. Let $\theta_0$ be as in Theorem \ref{engulfthm} and $K$ as in Lemma \ref{KKlem}. Note that, by Lemma \ref{KKlem} and $K=2\theta_0^2>4\theta_0$, we have
$$S_u(x, 4\theta_0\bar h(x))\subset S_u(x, K\bar h(x))\subset S_u(0, \hat K h)\subset\subset\Omega.$$
By the Vitali covering Lemma \ref{Vita_cov}, we can select a 
subcollection of disjoint sections $S_j := S_u(x_j,\bar h(x_j))$ such that $$F \subset \bigcup_{j=1}^{\infty} S_u(x_j, K\bar h(x_j)).$$ 
The volume estimates in Lemma \ref{vol-sec1} then imply that, for each $j$, we have $$|S_u(x_j, K\bar h(x_j))|\leq
C(n,\lambda,\Lambda) |S_u(x_j,\bar h(x_j))|.$$ 
By construction, $S_j \subset F$ and $|S_j \cap E| \leq (1-\delta) |S_j|$. Thus, we have that $|S_j \cap (F \setminus E)| \geq \delta |S_j|$. Therefore
\begin{align*}
|F \setminus E| \geq \sum_{j=1}^{\infty} |S_j \cap (F \setminus E)| 
 &\geq \sum_{j=1}^{\infty} \delta |S_j| \\
& \geq \frac{\delta}{C(n,\lambda,\Lambda)} \sum_{j=1}^\infty|S_u(x_j, K\bar h(x_j))|  \geq \frac{\delta}{C(n,\lambda,\Lambda)} |F|.
\end{align*}
Hence $|E|\leq (1-c_2\delta)|F|$ where $c_2 = C(n,\lambda,\Lambda)^{-1}$.
\end{proof}

\section{Measure estimate and sliding paraboloids from below}
\label{measure_sec}
In this section, we prove the measure estimate for supersolution of (\ref{LMAeq}) by sliding paraboloids from below.
Our measure estimate states as follows.
\begin{lem}[Measure estimate]\label{meas_lem}  Assume that (\ref{pinch1}) and (\ref{unilamU}) are satisfied in $\Omega$. 
 Suppose that $v\geq 0$ is a $W^{2, n}_{\loc}(\Omega)$ solution of 
 \begin{equation}a^{ij}v_{ij} + b\cdot Dv + cv\leq f
  \label{meas_eq}
 \end{equation}
 in a normalized section $S_u(0, 4t_0)\subset\subset\Omega$. 
 There are small, universal constants $\delta_1>0,\alpha_1(n,\lambda,\Lambda)>0, \e_1>0$ and a large, universal constant $M_1(n,\lambda,\Lambda)>1$ with the following properties. If 
 $$\inf_{S_u(0, \alpha_1 t_0)}v\leq 1$$ 
 and $$\|b\|_{L^n(S_u(0, 4t_0))}+ \|c^{-}\|_{L^n(S_u(0, 4t_0))} + \|f^{+}\|_{L^n(S_u(0, 4t_0))}\leq \e_1$$
 then 
 $$|\{v>M_1\}\cap S_u(0, t_0)|\leq (1-\delta_1) |S_u(0, t_0)|.$$
 \end{lem}
\begin{proof}[Proof of Lemma \ref{meas_lem}] Since $B_1(0)\subset S_u(0, 4t_0)\subset B_n(0)$, and  $S_u(0, 4t_0)\subset\subset\Omega$, the volume estimates in Lemma 
\ref{vol-sec1} give that 
$$C_1(n,\lambda,\Lambda)^{-1}\leq t_0\leq C_1(n,\lambda,\Lambda).$$
Our proof below uses this range of $t_0$ but not its specific value, so we can assume without loss of generality that $t_0=1$.
Denote for simplicity
 $S_t=S_u(0, t) ~\text{for } t>0.$\\
 Suppose $v(x_0)\leq 1$ at $x_0\in S_{\alpha_1}$ where $\alpha_1\in (0, 1/8)$. Consider the set of vertices $V=S_{\alpha_1}$. 
 {\it For each $y\in V$, we slide 
 the generalized paraboloids $ -a [u(x) -D u(y)\cdot(x-y) - u(y)] + C_y$ of opening $a>0$ until they touch
 the graph of $v$ from below at some point $x\in \overline S_1$, called the contact point.}
 In terms of the notation first introduced in \cite{Sa}, we define the contact set by
\begin{multline*}A_a (V, u, S_1, v)=\{ x\in \overline{S_1}: \text{there is } y\in V~\text{such that } \inf_{\overline{S_1}} \left(v + a [u(\cdot)- Du(y)\cdot (\cdot-y) - u(y)]\right)\\= v(x)
+ a[u(x)-Du(y)\cdot (x-y)- u(y)]\}.
\end{multline*}
{\bf Claim.} There exists a large, universal constant 
 $a(n,\lambda,\Lambda)$ such that
$$A_a (V, u, S_1, v)\subset S_1.$$
Indeed, for each $y\in V$, we consider the function
 $$P(x) = v(x) + a [u(x) -Du(y)\cdot(x-y) - u(y)]$$
 and look for its minimum points on $\overline{S_1}$. 
 Because $y\in S_{\alpha_1}\subset S_{1/8}$, we can use Theorem \ref{pst} to get a universal constant $c_1(n,\lambda,\Lambda)$ such that $S_u(y, c_1)\subset S_1$. Therefore, 
 if $x\in\p S_1$, then we have
 \begin{equation}P(x)\geq a [u(x) -D u(y)\cdot(x-y) - u(y)] \geq a c_1(n,\lambda,\Lambda).
  \label{PbdrS1}
 \end{equation}
Let $\theta_0(n,\lambda,\Lambda)$ be the universal constant in Theorem \ref{engulfthm}. Then,
at $x_0$, we have
 \begin{equation}P(x_0)\leq 1 + a [u(x_0) -Du(y)\cdot(x_0-y) - u(y)] \leq 1 + a \alpha_1\theta_0.
\label{Px0}  
 \end{equation}
The last inequality follows from the engulfing property. Indeed, we have $x_0, y\in S_{\alpha_1}$ and hence by the engulfing property of sections in Theorem
 \ref{engulfthm}, $x_0, y\in S_u(0,\alpha_1)\subset 
 S_u(y,\theta_0 \alpha_1)$.
 Consequently, $$u(x_0) -Du(y)\cdot(x_0-y) - u(y) \leq \theta_0\alpha_1.$$
Fix $\alpha_1(n,\lambda,\Lambda)>0$ small, universal and $a, M_1$ large depending  only on $n,\lambda$ and $\Lambda$, such that
 \begin{equation}M_1= 2 + a \alpha_1\theta_0 < a c_1.
  \label{M1eq}
 \end{equation}
Then, from (\ref{PbdrS1}) and (\ref{Px0}), we deduce that $P$ attains its minimum on $\overline{S_1}$ at a point $x\in S_1$. 
Hence  $A_a (V, u, S_1, v)\subset S_1$, proving the Claim.

In the above argument, we find from (\ref{Px0}) and (\ref{M1eq}) that
 \begin{equation}v(x) \leq P(x_0)< M_1~\text{for all}~x\in  A_a (V, u, S_1, v).
  \label{vupperM}
 \end{equation}
 Now, for each vertex $y\in V$, we  look at
the contact point $x\in S_1$, that is
$$ v(x)
+ a[u(x)-Du(y)\cdot (x-y)- u(y)]\leq v(z)
+ a[u(z)-Du(y)\cdot (z-y)- u(y)]~\text{for all~} z\in\overline{S_1}.$$
Then
$$D v(x)= a (Du(y)-D u(x))$$
which gives
\begin{equation}D u(y) = D u(x) +\frac{1}{a}D v(x)
 \label{uuv}
\end{equation}
and, from the gradient estimate in Lemma \ref{slope-est}, 
\begin{equation}
 \label{vxbound}
 |Dv(x)|\leq a C(n,\lambda,\Lambda).
\end{equation}
From the minimality of $P$ at $x$, we have 
\begin{equation}
D^2 v(x) \geq -a D^2 u(x).
\label{uvD2}
\end{equation}
Hence, upon differentiating (\ref{uuv}) with respect to $x$, we find
\begin{equation}D^2 u(y) D_x y = D^2 u(x) +\frac{1}{a} D^2 v(x)\geq 0.
 \label{uvD22}
\end{equation}
Using (\ref{unilamU}) and (\ref{uvD2}) at $x$, we  have 
$$\tilde\lambda U^{ij} (v_{ij} + a u_{ij})\leq a^{ij}(v_{ij} + a u_{ij})$$
from which we deduce that
$$na\tl + \tl \trace((D^2 u)^{-1}D^2 v) \leq (\det D^2 u(x))^{-1}[a^{ij} v_{ij} + a a^{ij} u_{ij}]\leq  (\det D^2 u(x))^{-1}a^{ij} v_{ij} + an\tL.$$
It follows that
$$ \trace((D^2 u)^{-1}D^2 v) \leq \tl^{-1} (\det D^2 u(x))^{-1} a^{ij} v_{ij} + an [\tL/\tl -1].$$
Now using the equation (\ref{meas_eq}) only at $x$, 
that is
$$\trace(AD^2 v)\leq f-b\cdot Dv -cv\leq f^{+}(x) + |b(x)||Dv(x)| + c^{-}(x) v(x),$$
and recalling $\det D^2 u(x)\geq\lambda$ by (\ref{pinch1}), we find that
\begin{eqnarray*}\trace ((D^2 u)^{-1} D^2 v(x))\leq \tl^{-1}
\lambda^{-1}\left(|b(x)| |Dv(x) | + c^{-}(x) v(x) + f^{+}(x)\right) + an [\tL/\tl -1].
 \end{eqnarray*}
Combining this with (\ref{uvD2}) and the basic estimates in Lemma \ref{LA_lem} (b), we find
\begin{eqnarray} -a D^2 u(x)\leq D^2 v(x)&\leq& \left[a C(n,\tl,\tL) + \tl^{-1}\lambda^{-1}\left(|b(x)| |Dv(x) | + c^{-}(x) v(x) + f^{+}(x)\right)\right] D^2 u(x) \nonumber\\ &\leq& 
a C(n, \lambda,\Lambda, \tl,\tL) (1 + |b(x)| + c^{-}(x) + f^{+}(x)) D^2 u(x).
 \label{uvupdown2}
\end{eqnarray}
Now, taking the determinant in (\ref{uvD22}) and invoking (\ref{uvupdown2}), we obtain
\begin{eqnarray*}\det D^2 u (y)\det D_x y&=& \det (D^2 u(x) +\frac{1}{a} D^2 v(x))\\ &\leq& C( n,\lambda,\Lambda, \tl,\tL)(1 + |b(x)|^n + |c^{-}(x)|^n + |f^{+}(x)|^n) \det D^2 u(x).
 \end{eqnarray*}
This together with (\ref{pinch1}) implies the bound 
$$0\leq \det D_x y \leq C(n,\Lambda,\lambda,\tl,\tL)(1 + |b(x)|^n + |c^{-}(x)|^n + |f^{+}(x)|^n).$$
Recall from (\ref{vupperM}) that, $v<M_1$ on the set $E:=A_a (V, u, S_1, v)$ of contact points. Moreover, by definition, $E$ is a closed set and thus measurable.
By the area formula, we have
\begin{eqnarray*}|S_{\alpha_1}|=\abs{V}= \int_{E}\abs{\det D_x y} &\leq& C(n,\Lambda,\lambda,\tl,\tL)\left(\abs{E} + \|b\|^n_{L^n(E)}+ \|c^{-}\|^n_{L^n(E)} + \|f^{+}\|^n_{L^n(E)} \right)
 \\&\leq& C(n,\lambda,\Lambda,\tl,\tL) |\{v<M_1\}\cap S_1| + C(n,\lambda,\Lambda,\tl,\tL) \e_1^n.
 \end{eqnarray*}
If $\e_1$ is small, universal, then $C(n,\lambda,\Lambda,\tl,\tL) \e_1^n\leq |S_{\alpha_1}|/2$ by the volume estimate of sections in Lemma \ref{vol-sec1} and thus
$$|S_{\alpha_1}| \leq 2C(n,\lambda,\Lambda,\tl,\tL) |\{v<M_1\}\cap S_1|.$$
Then, using the volume estimate of sections in Lemma \ref{vol-sec1}, we find that $$|S_1|\leq C^{\ast} |\{v<M_1\}\cap S_1|$$ for some $C^{\ast}>1$ universal. 
The conclusion of the Lemma holds with $\delta_1= 1/C^{\ast}.$
\end{proof}

\section{Doubling estimate and sliding paraboloids from below}
\label{double_sec}
In this section, we prove the doubling estimate for supersolution of (\ref{LMAeq}) by sliding paraboloids from below.
The key doubling estimate is the following lemma.
\begin{lem}[Doubling estimate]
\label{double_lem} Assume that (\ref{pinch1}) and (\ref{unilamU}) are satisfied in $\Omega$. 
Suppose that $v\geq 0$ is a $W^{2, n}_{\loc}(\Omega)$ solution of \begin{equation}a^{ij}v_{ij} + b\cdot Dv + cv
\leq f
                                                                   \label{doub_eq}
                                                                  \end{equation}
 in a normalized section $S_u(0, 4t_0)\subset\subset\Omega$.  
Let $\alpha\in (0, 1/8)$. There
is a small constant $\e_2$ depending only on $n,\lambda,\Lambda,\tl,\tL$ and $\alpha$ so that if
$$\|b\|_{L^n(S_u(0, 4t_0))}+ \|c^{-}\|_{L^n(S_u(0, 4t_0))} + \|f^{+}\|_{L^n(S_u(0, 4t_0))}\leq \e_2$$
 and if $v\leq 1$ at some point in $\overline {S_u(0, t_0)}$ then $v\leq M_2(n, \Lambda, \lambda,\tl,\tL, \alpha)$ in $S_u(0, \alpha t_0).$
\end{lem}

A simpler version of Lemma \ref{double_lem} is the following:
\begin{lem}[Doubling estimate for equations without lower order terms]
\label{double_lem1} Assume that (\ref{pinch1}) is satisfied in $\Omega$. Suppose that $v\geq 0$ is a $W^{2, n}_{\loc}(\Omega)$ solution of $U^{ij}v_{ij}\leq 0$ in a normalized 
section $S_4:= S_u(0, 4)\subset\subset\Omega$. Denote for simplicity
 $S_t=S_u(0, t)$ for $t>0$. Let $\alpha\in (0, 1/8)$.
If $v\geq 1$ in $S_\alpha$ then $v\geq c(n, \Lambda, \lambda,\alpha)$ in $S_{1}.$
\end{lem}
For reader's convenience, we present a proof of Lemma \ref{double_lem1} using the traditional construction of barriers. Another purpose in giving this proof is to motivate 
the introduction of the modification functions $h_{\e}$ and $h_\delta$ in the proof of Lemma \ref{double_lem}. We do not need these functions in the uniformly elliptic equations
 (such as (\ref{LMAeq}) with $u(x)=\frac{1}{2}|x|^2$) but they seem to be unavoidable in the degenerate and singular equations.

\begin{proof}[Proof of Lemma \ref{double_lem1}] Subtracting an affine function from $u$, we can assume that $u\geq 0$, $u(0)=0, Du(0)=0$. 
Recall that $$B_1(0)\subset S_u(0, 4)\subset B_n(0).$$
To prove the lemma, it suffices to construct a subsolution $w: S_{2}\backslash S_{\alpha}\longrightarrow \R$, i.e.,
$U^{ij}w_{ij}\geq 0$,
with the following properties:
$$w\leq 0 ~\text{on } \partial S_2,~w\leq 1 ~\text{on }\partial S_\alpha, w\geq c(n,\Lambda,\lambda,\alpha) ~\text{in } S_{1}\backslash S_{\alpha}.$$
Our first guess is
$$w= C(\alpha, m) (u^{-m}- 2^{-m})$$
where $m$ is large, depending only on $n,\lambda,\Lambda$ and $\alpha$. 

Let $(u^{ij})_{1\leq  i, j\leq n}$ be the inverse matrix $(D^2 u)^{-1}$ of the Hessian matrix $D^2 u$. 
We can compute for $W= u^{-m}-2^{-m}$
\begin{equation}u^{ij}W_{ij} = m u^{-m-2}[(m+1) u^{ij}u_{i}u_{j}-u u^{ij} u_{ij}] = mu^{-m-2}[(m+ 1) u^{ij} u_{i}u_{j}- n u].
 \label{uW1}
\end{equation}
By Lemma \ref{LA_lem}(c) 
\begin{equation}u^{ij}u_{i}u_{j} \geq \frac{\abs{D u}^2}{\trace (D^2 u)}.
 \label{D2uDu}
\end{equation}
If $x\in S_2\setminus S_{\alpha}$ and $y=0$ then from from the convexity of $u$, we have
$0=u(y)\geq u(x)+ D u(x)\cdot (0-x)$ and 
therefore,
\begin{equation}\abs{Du(x)}\geq \frac{u(x)}{\abs{x}}\geq \frac{\alpha}{n}\equiv 2c_0(n,\alpha)~\text{on}~S_2\setminus S_{\alpha}.
 \label{Duout}
\end{equation}
In view of (\ref{D2uDu}) and (\ref{Duout}), it is clear that in order to obtain $u^{ij}W_{ij}\geq 0$  using (\ref{uW1}), we only have trouble when $\|D^2u\|$ is unbounded. But the set of bad points, i.e., 
where $\|D^2 u\|$ is large, is small. Here is how we see this.
Because $S_u(0, 4)$ is normalized, we can deduce from Aleksandrov maximum principle, Theorem \ref{Alekmp} applied to $u-4$, that
$$\dist(S_u(0, 3), \p S_u(0, 4))\geq c_1(n,\lambda,\Lambda)$$ for some universal $c_1(n,\lambda,\Lambda)>0$. By Lemma \ref{slope-est},
$D u$ is bounded on $S_3$. Now let $\nu$ denote the outernormal unit vector field on $\p S_3$. Then, using the convexity of $u$, we have $\|D^2u\|\leq \Delta u$. Thus, by the divergence
theorem,
\begin{equation*}\int_{S_3}\|D^2 u\| \leq \int_{S_3}\Delta u =\int_{\partial S_3} \frac{\partial u}{\p \nu}\leq C(n, \lambda,\Lambda).
\end{equation*}
Therefore, given $\varepsilon >0$ small, the set
$$H_{\varepsilon} = \{x\in S_3\mid \|D^2 u\|\geq \frac{1}{\varepsilon}\}.$$
has measure bounded from above by
$$|H_{\e}|\leq C \e.$$
To construct a proper subsolution bypassing the bad points in $H_{\e}$,
we only need to modify $w$ at bad points. Roughly speaking, the modification involves the convex solution to 
$$\det D^2 u_{\e} = 2^n\Lambda\chi_{H_{\varepsilon}}~\text{in } S_4,~ u_{\e}=0 ~\text{on}~\partial S_4.$$
Here we use $\chi_E$ to denote the characteristic function of the set $E$: $\chi_E(x)=1$ if $x\in E$ and $\chi_E(x)=0$ if otherwise.
The problem with this equation is that the solution is not in general smooth while we need two derivatives to construct the subsolution. But this smoothness
problem can be fixed, using approximation as in the proof of the doubling estimate in Caffarelli-Guti\'errez \cite[Theorem 2]{CG}, as follows. 

We approximate $H_\e$ by an open set $\tilde H_\e$ where $H_\e\subset \tilde H_\e\subset S_4$ and the measure of their difference is small, that is
$$|\tilde H_\e\setminus H_\e|\leq \e.$$
We introduce a smooth function $\varphi$ with the following properties:
$$\varphi=1~\text{in}~H_\e,~\varphi=\e~\text{in}~S_4\setminus \tilde H_\e,~\e\leq \varphi\leq 1~\text{in } S_4.$$
Let $h_{\e}$ be the convex solution to
\begin{equation}\det D^2 h_{\e} = 2^n\Lambda\varphi ~\text{in } S_4,~ h_{\e}=0 ~\text{on}~\partial S_4.
 \label{he_eq}
\end{equation}
Note that $h_\e\leq 0$ in $S_4$. Since the right hand side $2^n\Lambda\varphi$ of the above equation is smooth and strictly positive, 
by the $C^{2,\alpha}$ estimates for the Monge-Amp\`ere equation (see, for example, Caffarelli \cite[Theorem 2]{C2}), we have $h_\e\in C^{2,\alpha}(S_4)$ for all $\alpha\in (0, 1)$.
From the Aleksandrov maximum principle, Theorem \ref{Alekmp}, we have on $S_4$
$$|h_{\e}|\leq C_n \diam (S_4) \left(\int_{S_4} 2^n\Lambda \varphi\right)^{1/n}.$$
We need to estimate the above right hand side.
From the definitions of $\tilde H_\e$ and $\varphi$, we can estimate
$$\int_{S_4}\varphi= \int_{H_\e} + \int_{\tilde H_\e\setminus H_\e}\varphi + \int_{S_4\setminus \tilde H_\e}\e
\leq |H_\e| +  |\tilde H_\e\setminus H_\e| + \e C(n,\lambda,\Lambda)\leq C(n,\lambda,\Lambda)\e.$$
It follows that for some universal constant $C_1(n,\lambda,\Lambda)$, 
\begin{equation}|h_\e|\leq C_1(n,\lambda,\Lambda)\e^{1/n}~\text{on~} S_4.
 \label{uph_eq}
\end{equation}
By the gradient estimate in Lemma \ref{slope-est}, we have on $S_2$
\begin{equation}|D h_{\e}(x)|\leq \frac{- h_{\e}(x)}{\dist(S_3, \p S_4)}\leq C_2(n,\lambda,\Lambda)\e^{1/n}.
 \label{upDh_eq}
\end{equation}
We choose $\e$ small so that 
\begin{equation}C_1(n,\lambda,\Lambda)\e^{1/n} \leq 1/4, ~C_2(n,\lambda,\Lambda)\e^{1/n}\leq c_0(n,\alpha).
 \label{epchoice}
\end{equation}
Let $$\tilde V = (u - h_{\e})~ \text{and} ~\tilde W = \tilde V^{-m}- 2^{-m}.$$
Then, from (\ref{Duout}), (\ref{uph_eq})-(\ref{epchoice}), we have
\begin{equation}
 \label{gradV} |\tilde V|\leq 3~\text{and}~
 |D \tilde V |\geq c_0(n,\alpha)~\text{on}~ S_2\setminus S_{\alpha};~\alpha \leq \tilde V \leq 1 + 1/4 <5/4~\text{on}~S_1\backslash S_{\alpha}.
\end{equation}
Now, compute as before
\begin{eqnarray*}u^{ij}\tilde W_{ij}& =& m \tilde V^{-m-2}[(m+1) u^{ij}\tilde V_{i}\tilde V_{j} - \tilde V u^{ij} \tilde V_{ij}]\\&=&
m \tilde V^{-m-2}[(m+1) u^{ij}\tilde V_{i}\tilde V_{j} + \tilde V (u^{ij} (h_{\e})_{ij}-n)].
\end{eqnarray*}
We note that, by Lemma \ref{LA_lem} (a), and (\ref{he_eq}),
$$u^{ij} ({h_\e})_{ij} = \trace ((D^2 u)^{-1}D^2 h_{\e})\geq n (\det (D^2 u)^{-1} \det D^2 h_{\e})^{1/n}\geq 2n~\text{on}~H_{\e}.$$
It follows that
$u^{ij}\tilde W_{ij}\geq nm \tilde V^{-m-1}~\text{on}~H_{\e}.$
On $(S_2\setminus S_{\alpha})\backslash H_{\e}$, we have $\trace (D^2 u)\leq n\e^{-1}$. Thus, from (\ref{gradV}) and Lemma \ref{LA_lem} (c), we have
\begin{eqnarray*}u^{ij}\tilde W_{ij} &\geq& m \tilde V^{-m-2}[(m+1) u^{ij}\tilde V_{i}\tilde V_{j}  -n\tilde V ]\\
 &\geq& m \tilde V^{-m-2}[(m+1) \frac{|D \tilde V|^2}{\trace (D^2 u)}-n\tilde V]\\
 &\geq&  m \tilde V^{-m-2} [(m+1)n^{-1}\e c^2_0(n,\alpha)-n\tilde V] \geq 0 
 \end{eqnarray*}
if we choose $m$ large, depending only on $n,\lambda,\Lambda$ and $\alpha$. Therefore, 
$u^{ij}\tilde W_{ij} \geq 0~\text{on}~S_2\setminus S_{\alpha}$
and hence $\tilde W = \tilde V^{-m}- 2^{-m}$ is a subsolution to $u^{ij} v_{ij}\geq 0$ 
on $S_2\setminus S_{\alpha}$.

Finally, by (\ref{gradV}) and $\tilde W\leq 0$ on $\p S_2$, we choose 
a suitable $C(\alpha, n,\lambda,\Lambda)$ so that 
the subsolution of the form
$$\tilde w = C(\alpha, n, \lambda,\lambda) (\tilde V^{-m}-  2 ^{-m})$$
satisfies $\tilde w \leq 1$ on $\p S_{\alpha}$. Now, 
we obtain the desired universal lower bound for $v$ in $S_{1}$ from $v\geq \tilde w$ on $S_1\backslash S_{\alpha}$ and $v\geq 1$ on $S_{\alpha}$.
 \end{proof}

\begin{proof}[Proof of Lemma \ref{double_lem}] As in the proof of Lemma \ref{meas_lem}, we can assume that $t_0=1$. 
Denote for simplicity
 $S_t=S_u(0, t) ~\text{for } t>0.$ Then $B_1(0)\subset S_u(0, 4)\subset B_n(0).$
Subtracting an affine function from $u$, we can assume that $$u\geq 0, u(0)=0, Du(0)=0.$$ 

 We argue by contradiction. Assume that $v(x_0)=1$ at some point $x_0\in\overline{S_1}$ and that $v>M_2$ everywhere in $S_{\alpha}$ for some
large $M_2$ to be determined. 
Because $S_u(0, 4)$ is normalized, we can deduce from the Aleksandrov maximum principle, Theorem \ref{Alekmp} applied to $u-4$, that
$$\dist(S_u(0, 3), \p S_u(0, 4))\geq c_1(n,\lambda,\Lambda)$$ for some universal $c_1(n,\lambda,\Lambda)>0$. 
As in the proof of Lemma \ref{double_lem1}, for each
$\delta >0$, the set
$$H_{\delta} = \{x\in S_3\mid \|D^2 u\|\geq \frac{1}{\delta}\}$$
has measure bounded from above by
$|H_{\delta}|\leq C(n,\lambda,\Lambda) \delta.$
As in the proof of Lemma \ref{double_lem1}, we can construct a convex function $h_{\delta}\in C(\overline{S_4})\cap C^{2}(S_4)$ with the following properties:
$$\det D^2 h_{\delta}= (1+2\tL/\tl)^n\Lambda~\text{in~} H_{\delta},~h_\delta = 0~\text{on}~\p S_4,~h_\delta\leq 0~\text{in}~S_4,$$
and, if $\delta$ is small, depending only on $n,\lambda,\Lambda,\tl,\tL$ and $\alpha$, we have
 \begin{equation} \label{delchoice}|h_{\delta}|\leq C_1(n,\lambda,\Lambda,\tl,\tL)\delta^{1/n} \leq \frac{\alpha}{8}~\text{in~} S_4,~|D h_{\delta}|\leq 
 C_2(n,\lambda,\Lambda,\tl,\tL)\delta^{1/n}\leq \frac{\alpha}{8n}~\text{in~} S_3.
   \end{equation}
Choose small, positive numbers $\beta=\beta(n,\lambda,\Lambda,\alpha)$ and $\e$, and a large number $M_2$, depending only on $n,\lambda,\Lambda,\tl,\tL$ and $\alpha$, with the following properties:
\begin{myindentpar}{1cm}
 (i) If $y\in S_\beta$ then 
 $|Du(y)|\leq \frac{\alpha}{16 n}.$\\
 (ii) $0<\e<\min\{ \log_2 (\frac{4}{3(1+\alpha)}),\frac{\delta\alpha^2}{32(1+\tL/\tl)n^4}\}.$\\
 (iii) $M_2^{\e}\geq \frac{16}{9\alpha}.$
\end{myindentpar}
Note that (i) is possible due to Theorem \ref{sec-size} and the $C^{1,\alpha_\ast}$ estimate for $u$ in Theorem \ref{C1alpha}. Indeed, if $y\in S_\beta$ then
the estimate on the size of sections in Theorem \ref{sec-size} gives $|y|\leq C(n,\lambda,\Lambda) \beta^{\mu}$.  
By the $C^{1,\alpha_\ast}$ estimate for $u$, 
$$|Du(y)|\leq C|y|^{\alpha_\ast}\leq C|\beta|^{\mu\alpha_\ast}\leq \frac{\alpha}{16n}$$
if $\beta$ is small, depending only on $n,\lambda,\Lambda$ and $\alpha$.\\
{\it For each $y\in S_\beta$, we slide 
 the generalized paraboloids  (modified by $h_{\delta}$) $ -\frac{3}{4} [u(x) -D u(y)\cdot(x-y) - u(y)- h_\delta (x)] + C_y$ of opening $\frac{3}{4}$ until they touch
 the graph of $-[v + 1]^{-\e}$ from below at some point $x\in \overline S_3$.} This is equivalent to looking for the maximum point in $\overline S_3$ of the 
 function $Q_y$ below.
 
Consider
$w(x)= [v(x) + 1]^{-\e}$
and for $y\in S_\beta$,
 $$Q_y(x) = w(x) -  \frac{3}{4}[u(x) -D u(y)\cdot (x-y) - u(y)-h_{\delta}(x)]\equiv [v(x) + 1]^{-\e}-  \frac{3}{4}[u(x) -Du(y)\cdot(x-y) - u(y)-h_{\delta}(x)].$$
Let $x\in \overline{S_3} $ be a maximum point of $Q_y$ on $\overline{S_3}$. Then, we have the following claims.\\
{\bf Claim 1.} $x\in S_3 \setminus S_{\alpha}$. \\
{\bf Claim 2.} $ \e [v(x) + 1]^{-\e -1} \trace ((D^2 u)^{-1}D^2 v(x)) \geq n\tL/\tl.$\\
{\bf Claim 3.} 
$[v(x) +1]^{\e}\leq \alpha^{-1}.$\\
{\bf Claim 4.} $|b(x)| + |c^{-}(x)| + |f^{+}(x)|\geq c_1(n,\lambda,\Lambda,\alpha,\tl,\tL)$ for some positive, small constant $c_1(n,\lambda,\Lambda,\alpha,\tl,\tL)$.\\
{\bf Claim 5.}
$|\det D_x y|\leq C(n,\ld,\Ld,\tl,\tL,\alpha)[|b(x)|^n + |c^{-}(x)|^n + |f^{+}(x)|^n].$\\
Given these claims, we finish the proof of the lemma as follows.
Denote the contact set by $E$. Then, since $y\in S_{\beta}$, we have from Claim 5 and the area formula,
$$|S_{\beta}|\leq \int_E |\det D_x y|\leq C\left (\|b\|^n_{L^n(E)}+ \|c^{-}\|^n_{L^n(E)} + \|f^{+}\|^n_{L^n(E)} \right)\leq C(n,\ld,\Ld,\tl,\tL,\alpha)\e_2^n.$$
If $\e_2$ is small, depending only on $n,\lambda,\Lambda,\tl,\tL$ and $\alpha$, then $C(n,\ld,\Ld,\tl,\tL,\alpha)\e_2^n< |S_{\beta}|/2$ by the volume estimate of 
sections in Lemma \ref{vol-sec1}. Thus the above inequalities
yield a contradiction. The lemma is proved.\\
We now proceed with the proofs of the claims.\\
{\it Proof of Claim 1}.
Claim 1 follows from the following observations:
\begin{myindentpar}{1cm}
 (a) For all $x\in \p S_3$, we have $Q_y(x)< Q_y(x_0).$\\
 (b)  For all $x\in S_{\alpha}$, we have $Q_y(x)< Q_y(x_0).$
\end{myindentpar}
Let us consider (a) which is equivalent to 
\begin{equation}
\label{bdrS3}
 [v(x)+ 1]^{-\e}- [v(x_0) + 1]^{-\e}\leq \frac{3}{4} [u(x)- u(x_0) - Du(y)\cdot(x-x_0) - h_{\delta}(x) + h_\delta(x_0)]~\text{for all}~x\in\p S_3.
\end{equation}
Suppose $x\in\p S_3$. Then, since $x, x_0\in S_4 \subset B_n(0),$
$|x-x_0|\leq 2n.$ By (\ref{delchoice}) and (i), (\ref{bdrS3}) follows from 
$$u(x)- u(x_0) - Du(y)\cdot(x-x_0) - h_{\delta}(x) + h_\delta(x_0)\geq 2-|Du(y)||x-x_0|  -\frac{\alpha}{8}\geq 2-\frac{\alpha}{2}\geq \frac{3}{2}\geq \frac{4}{3}[v(x)+ 1]^{-\e}.$$
Let us now prove (b) by showing that for all $x\in S_{\alpha}$ the following inequality
holds
\begin{equation}
 \frac{3}{4}[u(x_0)- u(x)- Du(y)\cdot(x_0-x)- h_{\delta}(x_0) + h_\delta (x)]< [v(x_0)+ 1]^{-\e}- [v(x) + 1]^{-\e}.
 \label{alphahalf}
\end{equation}
Suppose $x\in S_{\alpha}$. We estimate the left hand side of (\ref{alphahalf}) from above, using (\ref{delchoice}) and (i), by
\begin{equation}\frac{3}{4}[u(x_0)- u(x)- Du(y)\cdot (x_0-x)- h_{\delta}(x_0) + h_\delta (x)] \leq \frac{3}{4}[1+  |Du(y)| 2n + \frac{\alpha}{8}]\leq \frac{3}{4}(1+\frac{\alpha}{4}).
\label{outalpha}
\end{equation}
From (ii) and (iii), we find that
$$2^{-\e}\geq \frac{3}{4} (1+\alpha)= \frac{3}{4}(1+\frac{\alpha}{4}) + \frac{9\alpha}{16} \geq \frac{3}{4}(1+\frac{\alpha}{4}) + M_2^{-\e} > (M_2 +1)^{-\e} + \frac{3}{4}(1+\frac{\alpha}{4}).$$
Thus, owing to (\ref{outalpha}) and $v(x_0)=1$, (\ref{alphahalf}) easily follows from the following estimates for all $x\in S_{\alpha}$:
$$\frac{3}{4}(1+\frac{\alpha}{4}) \leq 2^{-\e}- [M_2 + 1]^{-\e}\leq [v(x_0)+ 1]^{-\e}- [v(x) + 1]^{-\e}.$$
{\it Proof of Claim 2}.
At the maximum point $x\in S_3\setminus S_{\alpha}$ of $Q_y$, we have the following information:
\begin{myindentpar}{1cm}
 (c) $D w(x) = \frac{3}{4}[Du(x)- Du(y) -D h_{\delta}(x)],$\\
 (d) $D^2 w(x)\leq \frac{3}{4}[D^2 u(x)-D^2 h_{\de}(x)].$
\end{myindentpar}
Equation (c) gives
$Du(y) = Du(x)-D h_{\de}(x)-\frac{4}{3} Dw(x).$
Differentiating this expression with respect to $x$ yields
\begin{equation}D^2 u(y) D_x y = D^2 u(x)-D^2 h_{\de}(x)-\frac{4}{3}D^2 w(x).
 \label{Dxy}
\end{equation}
Moreover, by convexity, (i) and (\ref{delchoice})
\begin{equation}|Dw(x)|\geq \frac{3}{4}\left(\frac{u(x)}{|x|}- |Du(y)|-|D h_\de (x)|\right)\geq \frac{3}{4}\left(\frac{\alpha}{n}-\frac{\alpha}{16n}-\frac{\alpha}{8n}\right)\geq \frac{\alpha}{2n}.
 \label{Dwbelow}
\end{equation}
Furthermore, by the gradient estimate for $u$ in Lemma \ref{slope-est} and (\ref{delchoice}), we also have an upper bound for $Dw(x)$
\begin{equation}
|Dw(x)|\leq C(n,\lambda,\Lambda).
 \label{Dwabove}
\end{equation}
From $w(x)= [v(x) + 1]^{-\e}$, we find
\begin{equation}Dw(x) = -\e [v(x) + 1]^{-\e-1} Dv(x)
 \label{Dwu}
\end{equation}
and, with $a\otimes b$ denoting the matrix $\left(a_i b_j\right)_{1\leq i, j\leq n}$ for $a= (a_1,\cdots, a_n)\in \R^n$ and $b= (b_1,\cdots, b_n)\in \R^n$, 
\begin{equation}D^2 w(x)= \e(\e+ 1) [v(x) + 1]^{-\e-2} Dv(x)\otimes Dv(x) -\e[v(x) +1]^{-\e-1} D^2 v(x).
 \label{D2wu}
\end{equation}
It follows from inequality (d) that
$$u^{ij}(-w_{ij})\geq \frac{3}{4}\trace ((D^2 u)^{-1}D^2 h_{\de}(x)) + \frac{3}{4}u^{ij}(-u_{ij})= \frac{3}{4}\trace ((D^2 u)^{-1}D^2 h_{\de}(x)) -\frac{3}{4}n.$$
Combining this estimate with (\ref{D2wu}), we obtain
\begin{multline}\e [v(x) + 1]^{-\e-1} \trace ((D^2 u)^{-1}D^2 v(x))\\  \geq \e (\e+1) [v(x) + 1]^{-\e-2}u^{ij}v_i v_j + \frac{3}{4}\trace ((D^2 u)^{-1}D^2 h_{\de}(x))-\frac{3}{4}n.
\label{traceD2uv}
 \end{multline}
There are two cases:\\
{\bf Case 1:} $x\in H_{\delta}$. In this case, by Lemma \ref{LA_lem} (a)
$$\trace ((D^2 u)^{-1}D^2 h_{\de}(x))-n \geq n [\det (D^2 u(x))^{-1} \det D^2 h_{\delta}(x)]^{1/n}-n \geq (1+2\tL/\tl)n-n=2n\tL/\tl$$
and Claim 2 easily follows from this estimate and (\ref{traceD2uv}).\\
{\bf Case 2:} $x\in (S_3\setminus S_{\alpha})\setminus H_{\delta}$. Then $\|D^2 u(x)\|\leq \frac{1}{\de}$ and hence $\trace (D^2 u(x))\leq n\delta^{-1}$. Using Lemma \ref{LA_lem} (c), (\ref{Dwu}) and (\ref{Dwbelow}), we find
\begin{eqnarray*}
 \e (\e+1) [v(x) + 1]^{-\e-2}u^{ij}v_i v_j-\frac{3}{4}n &\geq&  \e (\e+1) [v(x) + 1]^{-\e-2} \frac{|Dv(x)|^2}{\trace (D^2 u(x))} -n\\
 &\geq&\frac{\de}{n} \e (\e+1) [v(x) + 1]^{-\e-2}|Dv|^2 -n
 \\ &= & \frac{\de (\e+1)}{n\e} [v(x) +1]^{\e} |Dw(x)|^2-n\\
 &\geq & \frac{\de }{n\e} |Dw(x)|^2-n \geq \frac{\de}{n\e}\frac{\alpha^2}{4 n^2}-n\geq n\tL/\tl.
\end{eqnarray*}
The last inequality follows from (ii). In this case, by recalling (\ref{traceD2uv}), we see that Claim 2 also follows.
In all cases, we have Claim 2.\\
{\it Proof of Claim 3}.
By the maximality of $Q_y(x)$, we have $Q_y(x)\geq Q_y(x_0)$. Therefore, recalling $v(x_0)=1$, we have
\begin{eqnarray}[v(x) + 1]^{-\e}&\geq& [v(x_0) + 1]^{-\e}-\frac{3}{4}[u(x_0) - u(x) + Du(y) \cdot(x-x_0) -h_\de (x_0) + h_\de (x)]
\nonumber\\ &\geq & 2^{-\e} -\frac{3}{4}[u(x_0) - u(x) + Du(y)\cdot (x-x_0) - h_\de (x_0)].
\label{vupper}
\end{eqnarray}
Recalling (i), (\ref{delchoice}), $x_0\in S_1$ and Claim 1, we can estimate
$$u(x_0) - u(x) + Du(y) \cdot(x-x_0) - h_\de (x_0) \leq 1-\alpha + |Du(y)|2n + \frac{\alpha}{8}\leq 1-\frac{\alpha}{2}.$$
Then, invoking (\ref{vupper}) and (ii), we find that Claim 3 follows from the estimates
$$[v(x) + 1]^{-\e} \geq 2^{-\e} - \frac{3}{4}(1-\frac{\alpha}{2}) \geq \frac{3}{4} (1+\alpha)-\frac{3}{4} (1-\frac{\alpha}{2}) \geq \alpha.$$
{\it Proof of Claim 4}. From Claim 3, we deduce that $v(x)$ is bounded from above
\begin{equation}v(x)\leq C(n,\lambda,\Lambda,\tl,\tL,\alpha).
\label{upvx}
\end{equation}
Then, by (\ref{Dwu}), $|Dv(x)|$ is also bounded from above, because
\begin{equation}|Dv(x)|= \e^{-1} [v(x) +1]^{\e+1} |Dw(x)|\leq C(n,\lambda,\Lambda,\tl,\tL,\alpha).
\label{upDvx}
\end{equation}
By (d) and (\ref{D2wu}), we have
$$\frac{3}{4}D^2 u (x) + \e [v(x)+1]^{-\e-1} D^2 v(x)\geq 0.$$
From (\ref{unilamU}), we obtain at $x$
\begin{eqnarray*}a^{ij} (\frac{3}{4} u_{ij} (x) + \e [v(x)+1]^{-\e-1} v_{ij}(x))&\geq& \tl U^{ij}(\frac{3}{4} u_{ij} (x) + \e [v(x)+1]^{-\e-1} v_{ij}(x)) \\ &=& \tl \det D^2 u(x)\left[ \frac{3}{4}n + 
\e [v(x)+1]^{-\e-1}\trace ((D^2 u)^{-1}(x) D^2 v(x))\right]. 
\end{eqnarray*}
Moreover,
$$\frac{3}{4}a^{ij} (x) u_{ij}(x)\leq \frac{3}{4}\tL U^{ij}(x) u_{ij}(x)= \frac{3}{4} \tL n\det D^2 u(x).$$
Hence , from (\ref{pinch1}), we have
$$\e [v(x)+1]^{-\e-1}\trace ((D^2 u)^{-1}(x) D^2 v(x)) \leq \tl^{-1} \lambda^{-1}\e [v(x)+1]^{-\e-1} \trace (AD^2 v) +  \frac{3}{4}n[\tL/\tl-1].$$
It follows from Claim 2 that
\begin{eqnarray*}n\tL/\tl &\leq&  \e [v(x) +1]^{-\e -1}\trace ((D^2 u)^{-1}(x) D^2 v(x))
\\ & \leq& \tl^{-1} \lambda^{-1}\e [v(x)+1]^{-\e-1} \trace (AD^2 v) +  n[\tL/\tl-1].
 \end{eqnarray*}
Therefore,
$$n\tl \lambda \e^{-1} [v(x) +1]^{\e+1} \leq \trace (AD^2 v).$$
Now using the equation (\ref{doub_eq}) only at $x$, that is
\begin{equation}\trace(AD^2 v)\leq f-b\cdot Dv -cv\leq f^{+}(x) + |b(x)||Dv(x)| + c^{-}(x) v(x),
\label{double_eq2}
\end{equation}
we obtain
$$n\tl\lambda \e^{-1}\leq f^{+}(x) + |b(x)||Dv(x)| + c^{-}(x) v(x).$$
Therefore, Claim 4 follows from the upper bounds of $v(x)$ and $|Dv(x)|$ in (\ref{upvx}) and (\ref{upDvx}). \\
{\it Proof of Claim 5}.
From (d) and (\ref{unilamU}), we can use $-D^2 w\geq -\frac{3}{4} D^2 u$ at $x$ together with the estimate
$$\tilde\lambda U^{ij} (-w_{ij} +  \frac{3}{4}u_{ij})\leq a^{ij}(-w_{ij} +  \frac{3}{4}u_{ij}) \leq a^{ij}(-w_{ij}) +\tL \frac{3}{4} U^{ij} u_{ij} = a^{ij}(-w_{ij}) + \frac{3}{4}\tL n \det D^2 u(x)$$
to obtain
$$\frac{3}{4}n\tl + \tl \trace((D^2 u)^{-1}(-D^2 w)) \leq (\det D^2 u(x))^{-1}a^{ij} (-w_{ij}) +  \frac{3}{4}n\tL.$$
Recalling (\ref{D2wu}), we find
$-D^2 w(x)\leq \e [v(x) +1]^{-\e-1} D^2 v(x),$
and hence,
$$\frac{3}{4}n\tl + \tl \trace((D^2 u)^{-1}(-D^2 w)) \leq (\det D^2 u(x))^{-1} \e [v(x) +1]^{-\e-1} \trace(AD^2 v) +  \frac{3}{4}n\tL.$$
Since $\det D^2 u(x)\geq\lambda$ by (\ref{pinch1}), it follows that
\begin{eqnarray*}\trace((D^2 u(x))^{-1}(-D^2 w(x))) \leq  \tl^{-1}\lambda^{-1} \e [v(x) +1]^{-\e-1} \trace(AD^2 v) + n [\tL/\tl -1].
\end{eqnarray*}
Now using the equation (\ref{doub_eq}) only at $x$, or (\ref{double_eq2}), we find that
\begin{multline*}\trace((D^2 u(x))^{-1}(-D^2 w(x)))  \leq 
\tl^{-1} \e [v(x) + 1]^{-\e-1} \lambda^{-1} [ |b(x)| |Dv(x)| + c^{-}(x) |v(x)| + f^{+}(x)] + n [\tL/\tl -1].
\end{multline*}
By Lemma \ref{LA_lem}(b), we deduce that
\begin{eqnarray*}-D^2 w(x) &\leq& \left( n \tL/\tl+ \tl^{-1}\e [v(x) + 1]^{-\e-1} \lambda^{-1} [ |b(x)| |Dv(x)| + c^{-}(x) v(x) + f^{+}(x)]\right)D^2 u(x)
 \\ &\leq & C(n,\ld,\Ld,\tl,\tL,\alpha) [ |b(x)| + c^{-}(x) + f^{+}(x)] D^2 u(x),
\end{eqnarray*}
where we used Claim 4 in the last inequality. 
Recalling (\ref{Dxy}), we get
$$D^2 u(y) D_x y \leq C(n,\ld,\Ld,\tl,\tL,\alpha)[|b(x)| + c^{-}(x) + f^{+}(x)] D^2 u(x).$$
Taking the determinant of both sides, and using (\ref{pinch1}), we obtain Claim 5. 
\end{proof}

\begin{rem}
\label{largeD_rem}
 From (\ref{Dwbelow}) and (\ref{Dwu}), we find
 $$|Dv(x)|= \e^{-1}[v(x) +1]^{\e+1} |Dw(x)|\geq \frac{\e^{-1}\alpha}{2n}.$$
 Thus the above argument also applies to functions that satisfy the equations  (\ref{doub_eq}) only at points  where the gradient is large.
\end{rem}
We record the above remark in the following lemma.
\begin{lem} [Doubling estimate for degenerate and singular elliptic equations with unbounded drift that hold only where the gradient is large.]
\label{large_lem} Assume that (\ref{pinch1}) and (\ref{unilamU}) are satisfied in $\Omega$. Let $\alpha\in (0, 1/8)$.
There are
small constants $\delta_2, \e_2$ depending only on $n,\lambda,\Lambda,\tl,\tL$ and $\alpha$ with the following properties.
Suppose that $v\geq 0$ is a $W^{2, n}_{\loc}(\Omega)$ solution of 
$a^{ij}v_{ij} + b\cdot Dv + cv
\leq f
 $
in a normalized section $S_u(0, 4t_0)\subset\subset\Omega$, but only at points where the gradient of $v$ is large, that is $|Dv|\geq \delta_2$.  
If
$$\|b\|_{L^n(S_u(0, 4t_0))}+ \|c^{-}\|_{L^n(S_u(0, 4t_0))} + \|f^{+}\|_{L^n(S_u(0, 4t_0))}\leq \e_2$$
 and if $v\leq 1$ at some point in $\overline {S_u(0, t_0)}$ then $v\leq M_2(n, \Lambda, \lambda,\tl,\tL, \alpha)$ in $S_u(0, \alpha t_0).$
\end{lem}

\section{Rescaling and Power decay estimate}
\label{decay_sec}
In this section, we discuss rescaling of the linearized Monge-Amp\`ere equation (\ref{LMAeq}) using John's lemma
and prove a power decay estimate for supersolution of (\ref{LMAeq}).
\subsection{Critical density estimate}
 Combining Lemmas \ref{meas_lem} and \ref{double_lem}, we obtain the 
 following result: 
\begin{prop}[Critical density estimate with $L^n$ drift]
 \label{decay_rem} Assume that (\ref{pinch1}) and (\ref{unilamU}) are satisfied in $\Omega$. 
 Suppose that $v\geq 0$ is a $W^{2,n}_{\loc}(\Omega)$ solution of $$a^{ij}v_{ij} + b \cdot Dv + cv\leq f$$  in a normalized section $S_u(0, 4t_0)\subset\subset\Omega$. 
 There are small, universal constants $\delta>0,\e_3>0$ and a large constant $M>1$ with the following properties. If 
 $$\|b\|_{L^n(S_u(0, 4t_0))}+ \|c^{-}\|_{L^n(S_u(0, 4t_0))} + \|f^{+}\|_{L^n(S_u(0, 4t_0))}\leq \e_3$$
 and for some nonnegative integer $k$
 \begin{equation}|\{v>M^{k+1}\}\cap S_u(0, t_0)|> (1-\delta) |S_u(0, t_0)|
  \label{Mkeq}
 \end{equation}
then $v>M^k$ in $S_u(0, t_0)$.
\end{prop}
\begin{proof} Let $\alpha_1, \delta_1, \e_1, M_1$ be as in Lemma \ref{meas_lem}. We choose $\alpha$ in Lemma \ref{double_lem} to be $\alpha_1$ and let $\e_2, M_2$ be the
universal constants obtained from this lemma.
 Let $$\delta:=\delta_1;~ M:= M_1 M_2, ~\text{and} ~\e_3:=\min\{\e_1,\e_2\}.$$
 We first prove the proposition for $k=0$.
 Since $M_2>1$, the function $\tilde v:= v/M_2$ is a 
 $W^{2,n}_{\loc}(\Omega)$ solution of $$a^{ij}\tilde v_{ij} + b \cdot D\tilde v + c\tilde v\leq\tilde f$$  in a normalized section $S_u(0, 4t_0)\subset\subset\Omega$
 where $\tilde f:=f/M_2$ and 
 $$\|b\|_{L^n(S_u(0, 4t_0))}+ \|c^{-}\|_{L^n(S_u(0, 4t_0))} + \|\tilde f^{+}\|_{L^n(S_u(0, 4t_0))}\leq \e_3 \leq \e_1.$$
 Since (\ref{Mkeq}) holds for $k=0$, 
  $$|\{\tilde v>M_1\}\cap S_u(0, t_0)|> (1-\delta) |S_u(0, t_0)|=(1-\delta_1) |S_u(0, t_0)|,$$
  we conclude from Lemma \ref{meas_lem} that $\inf_{S_u(0, 
  \alpha_1 t_0)} \tilde v >1,$ so $v>M_2$ in $S_u(0,\alpha_1 t_0)$. Applying Lemma \ref{double_lem} to $v$, we obtain that $v>1$ in $S_u(0, t_0)$. Thus, the conclusion of the
  proposition holds for $k=0$.
Now, for $k\geq 1$, we apply the result in the case $k=0$ to $v/M^{k}$ to get the desired lower bound $v>M^k$ for $v$ on $S_u(0, t_0)$.
 \end{proof}

\subsection{Rescaling the linearized Monge-Amp\`ere equation } 
\label{AIP_sec}
Assume that (\ref{pinch1}) and (\ref{unilamU}) are satisfied in $\Omega$. 
In this section, we record how the equation
\begin{equation}a^{ij} v_{ij} + b\cdot Dv + c v\leq f
 \label{eqSh}
\end{equation}
changes with respect to normalization of a section $S_u(x_0, h)$ of $u$.
By subtracting $u(x_0) + Du(x_0)\cdot (x-x_0)+ h$ from $u$, we may assume 
that $u\mid_{\p S_u(x_0, h)} =0$ and $u$ achieves its minimum $-h$ at $x_0$. By John's lemma, we can 
find an affine transformation $Tx =A_h x+ b_h$ such that
\begin{equation}B_1 (0)\subset T^{-1} (S_u(x_0, h))\subset B_n(0).
 \label{normSh}
\end{equation}
Let 
$$\tilde u(x) =  (\det A_h)^{-2/n} u(Tx),~\tilde v(x)= v(Tx).$$
Then from (\ref{pinch1}), we have
\begin{equation}\lambda \leq \det D^2 \tilde u \leq \Lambda~\text{in}~T^{-1} (S_u(x_0, h)),
 \label{detD2tilde}
\end{equation}
$\tilde u =0$ on $\p T^{-1}(S_u(x_0, h))$ and
$$B_1(0)\subset \tilde S:=T^{-1}(S_u(x_0, h)) = S_{\tilde u}(y, (\det A_h)^{-2/n} h)\subset B_n(0)$$
where $y$ is the minimum point of $\tilde u$ in $T^{-1} (S_u(x_0, h))$. \\
The equation (\ref{eqSh}) 
becomes
\begin{eqnarray}
 \tilde a^{ij}\tilde v_{ij} + \tilde b\cdot D\tilde v  +  \tilde c \tilde v \leq \tilde f(x)
 \label{eqSh1}
\end{eqnarray}
where the coefficient matrix $\tilde A= \left(\tilde a^{ij}\right)_{1\leq i, j\leq n}$ is given by
\begin{equation}
 \label{tildeAA}
 \tilde A= (\det A_h)^{2/n} A_h^{-1} A (A_h^{-1})^{t}.
\end{equation}
and the lower order terms $\tilde b, \tilde c$ and $\tilde f$ are given respectively by
\begin{equation}\tilde b(x) =(\det A_h)^{2/n} A_h^{-1}b(Tx),~ \tilde c(x) =(\det A_h)^{2/n} c(Tx),~\tilde f(x) =(\det A_h)^{2/n} f(Tx).
 \label{bcf}
\end{equation}
For completeness, we include the standard computation leading to (\ref{eqSh1}). We have
$$D \tilde u = (\det A_h)^{-2/n} A_h^{t} Du;~D^2 \tilde u = (\det A_h)^{-2/n} A_h^{t}D^2 u A_h;~D \tilde v = A_h^{t} Dv,~D^2 \tilde v = A_h^{t}D^2 v A_h.$$
The cofactor matrix $\tilde U= (\tilde U^{ij})_{1\leq i, j\leq n}$ of $D^2 \tilde u$ is related to $U$ and $A_h$ by
\begin{eqnarray}\tilde U = (\det D^2 \tilde u)(D^2 \tilde u)^{-1}=(\det D^2 u) (\det A_h)^{2/n} A_h^{-1} (D^2 u)^{-1} (A_h^{-1})^{t}=
(\det A_h)^{2/n} A_h^{-1} U (A_h^{-1})^{t}. 
\label{tildeU}
\end{eqnarray}
Therefore, by (\ref{tildeAA}), $$\tilde a^{ij} \tilde v_{ij} (x)= \trace (\tilde A D^2\tilde v)= (\det A_h)^{2/n} \trace (AD^2 v(Tx))= (\det A_h)^{2/n}  a^{ij} v_{ij}(Tx)$$
and hence, recalling (\ref{eqSh}),
\begin{eqnarray*}
 \tilde a^{ij} \tilde v_{ij} (x) = (\det A_h)^{2/n}[f(Tx)- c(Tx) v(Tx)- b(Tx)\cdot Dv(Tx)]
 = \tilde f(x) - \tilde c \tilde v -\tilde b\cdot D\tilde v
\end{eqnarray*}
where $\tilde b, \tilde c, \tilde f$ are defined by (\ref{bcf}). Thus, we get (\ref{eqSh1}) as asserted.\\
We remark from (\ref{detD2tilde}), (\ref{tildeAA}) and (\ref{tildeU}) that $\tilde u$ and $\tilde A$ also satiisfy the structural conditions (\ref{pinch1}) and (\ref{unilamU}). This is
the affine invariance property of (\ref{LMAeq}) under rescaling.\\
Using the volume estimates in Lemma \ref{vol-sec1}, we find from (\ref{normSh}) that 
\begin{equation}[C(n,\lambda,\Lambda)]^{-1} h^{n/2}\leq \det A_h
\leq C(n,\lambda,\Lambda) h^{n/2}.
\label{detAh}
\end{equation}
The interior $C^{1,\alpha_\ast}$ estimate for $u$ in Theorem \ref{C1alpha} shows that 
 $$S_u(x_0, h)\supset B(x_0, c_1h^{\frac{1}{1+\alpha_\ast}})$$ for some universal constant $c_1= c_1(n,\lambda,\Lambda)$. This combined with (\ref{normSh}) implies that 
 \begin{equation} \|A_h^{-1}\|\leq Ch^{-\frac{1}{1+\alpha_\ast}}= Ch^{-1 +\frac{\alpha_\ast}{1+\alpha_\ast}}.
  \label{Anorm}
 \end{equation}
By (\ref{bcf}) and (\ref{Anorm}), we can estimate for all $p\geq 1$ 
\begin{eqnarray}\|\tilde b\|_{L^p(\tilde S)}\leq C h \| A_h^{-1}b(Tx)\|_{L^p(\tilde S)}\leq
Ch^{1-\frac{n}{2p}}\|A_h^{-1}\|\|b\|_{L^p(S_u(x_0, h))} \leq C h^{\frac{\alpha_\ast}{1+\alpha_\ast}-\frac{n}{2p}}\|b\|_{L^p(S_u(x_0, h))}.
\label{tbLp}
\end{eqnarray}
Suppose from now on $p\geq n$.  Then, by H\"older inequality, 
\begin{equation} \|\tilde b\|_{L^n(\tilde S)}\leq |\tilde S|^{\frac{1}{n}-\frac{1}{p}}\|\tilde b\|_{L^p(\tilde S)}\leq 
C h^{\frac{\alpha_\ast}{1+\alpha_\ast}-\frac{n}{2p}}\|b\|_{L^p(S_u(x_0, h))}.
\label{tbLn}
\end{equation}
Moreover, by (\ref{detAh}), we have
\begin{equation} \|\tilde c\|_{L^n(\tilde S)}=(\det A_h)^{2/n} h^{-1/2} \|c\|_{L^n(S_u(x_0, h))}\leq Ch^{1/2} \|c\|_{L^n(S_u(x_0, h))}
 \label{tcLn}
\end{equation}
and 
\begin{equation} \|\tilde f\|_{L^n(\tilde S)}=(\det A_h)^{2/n} h^{-1/2} \|f\|_{L^n(S_u(x_0, h))}\leq Ch^{1/2} \|f\|_{L^n(S_u(x_0, h))}.
 \label{tfLn}
\end{equation}
Suppose furthermore, $$\frac{\alpha_\ast}{1+\alpha_\ast}-\frac{n}{2p}> 0~ \text{or, equivalently}, ~p> \frac{n(1+\alpha_\ast)}{2\alpha_\ast}.$$
Then in our rescaling process, the $L^n$ norms of $b, c, f$ are under control if $h$ is bounded by a universal constant $c_\ast$ to be chosen as follows. By the volume
estimates of sections in Lemma \ref{vol-sec1}, we can find $c_{\ast}(n,\lambda,\Lambda)$ such that
$$ \text{if}~S_u(y, h)\subset B_n(x_0) ~\text{then } h\leq c_\ast.$$
We summarize the above discussion in the following lemma.
\begin{lem} [Critical density estimate with $L^p$ drift]
\label{res_lem}
 Assume that (\ref{pinch1}) and (\ref{unilamU}) are satisfied in $\Omega$. Let $p>\frac{n(1+\alpha_\ast)}{2\alpha_\ast}$.
There is a small number $\e_4$ depending only on $p,n,\lambda,\Lambda,\tl$ and $\tL$ with the following property.
 Suppose that $v\geq 0$ is a $W^{2, n}_{\loc}(\Omega)$ solution of  $$a^{ij}v_{ij} + b\cdot Dv + cv\leq f$$  in a section $S_u(x_0, h)\subset\subset \Omega$ with $h\leq 4 c_\ast$ and 
 \begin{equation}\|b\|_{L^p(S_u(x_0, h))}+ \|c^{-}\|_{L^n(S_u(x_0, h))} + \|f^{+}\|_{L^n(S_u(x_0, h))}\leq \e_4.
  \label{bLp}
 \end{equation}
Let $M$ and $\delta$ be as in Proposition \ref{decay_rem}. If
 for some nonnegative integer $k$, we have
 \begin{equation}|\{v>M^{k+1}\}\cap S_u(x_0, h/4)|> (1-\delta) |S_u(x_0, h/4)|
  \label{Mkeq_res}
 \end{equation}
then $v>M^k$ in $S_u(x_0, h/4)$.
\end{lem}
\begin{proof}[Proof of Lemma \ref{res_lem}]
 Rescaling the equation as above, we obtain (\ref{eqSh1}) on a normalized section $\tilde S$ of $\tilde u$, where $\tilde S:=T^{-1}(S_u(x, h))= S_{\tilde u}(y, 4t_0)$. If $\e_4$ is small,
 depending only on $p,n,\lambda,\Lambda,\tl$ and $\tL$ , then from (\ref{bLp}) and (\ref{tbLn})-(\ref{tfLn}), we find 
 $$\|\tilde b\|_{L^n(\tilde S)}+ \|\tilde c^{-}\|_{L^n(\tilde S)} + \|\tilde f^{+}\|_{L^n(\tilde S)}\leq \e_3.$$
 Note that
 $$T^{-1}(\{v>M^{k+1}\}\cap S_u(x_0, h/4)) = \{\tilde v> M^{K+1}\}\cap S_{\tilde u}(y, t_0); ~T^{-1}( S_u(x_0, h/4))= S_{\tilde u}(y, t_0).$$
 Thus, the assumption (\ref{Mkeq_res}) implies
$$|S_{\tilde u}(y, t_0) \cap \{\tilde v>M^{k+1}\}| > (1-\delta) |S_{\tilde u}(y, t_0)|.$$
Proposition \ref{decay_rem} applied to (\ref{eqSh1}) gives  $\tilde v>M^k$ on $S_{\tilde u}(y, t_0)$. Rescaling back, we have $v>M^k$ in $S_u(x_0, h/4)$. 
\end{proof}

\subsection{Power decay estimate}
From the critical density estimate and the covering lemma stated in Lemma \ref{inkspots}, we obtain the $L^{\e}$ estimate and completing the proof of the power decay estimate.
\begin{thm}[Decay estimate of the distribution function] \label{decay_thm} 
 Assume that (\ref{pinch1}) and (\ref{unilamU}) are satisfied in $\Omega$. Let $p>\frac{n(1+\alpha_\ast)}{2\alpha_\ast}$. Let $\hat K$ be as in Lemma \ref{KKlem}.
 Let $\e_4$ and $c_\ast$ be as in Lemma \ref{res_lem}.
 Suppose that $v\geq 0$ is a $W^{2,n}_{\loc}(\Omega)$ solution of $$a^{ij}v_{ij} + b \cdot Dv + cv\leq f$$  in a section $S_4=S_u(0, 4 t_0)\subset\subset\Omega$ with
 $S:= S_u(0, \hat K t_0)\subset\subset\Omega$, $t_0\leq c_*$ and 
 \begin{equation}\|b\|_{L^p(S)}+ \|c^{-}\|_{L^n(S)} + \|f^{+}\|_{L^n(S)}\leq \e_4.
  \label{bcf_small}
 \end{equation}
Suppose that
 $$\inf_{S_u(0, t_0)} v\leq 1.$$
Then there are universal constants $C_1(n,\lambda,\Lambda,\tl,\tL)>1$  and $\e(n,\lambda,\Lambda,\tl,\tL)\in (0, 1)$ such that
\[ |\{ v > t \} \cap S_u(0, t_0) | \leq C_1 t^{-\eps} |S_u(0, t_0)|~\text{for all~} t>0.\]
\end{thm}

\begin{proof}[Proof of Theorem \ref{decay_thm}] 
Let $\delta\in (0, 1)$ and $M>1$ be the constants in Proposition \ref{decay_rem}. The conclusion of the theorem follows from the following decay estimate
for  $A_k := \{v > M^k\} \cap S_u(0, t_0)$:
\[ |A_k| \leq C_2 M^{-\eps k}|S_u(0, t_0)|. \]
Note that $A_k$'s are open sets and $A_k \subset A_1$ for all $k\geq 1$. Recalling $\inf_{S_u(0, t_0)} v \leq 1$, by Lemma \ref{res_lem},
we have 
$$|A_k|\leq |A_1| \leq (1-\delta)|S_u(0, t_0)|~\text{for all }k.$$ 
{\bf Claim.} If  a section $S_u(y, t) \subset S_u(0, t_0)$ satisfies 
\begin{equation}|S_u(y, t) \cap A_{k+1}| > (1-\delta) |S_u(y, t)|,
 \label{upper_den}
\end{equation}
 then $S_u(y, t) \subset A_{k}.$
 
Indeed, from $S_u(y, t)\subset S_u(0, t_0)$ and Lemma \ref{KKlem}, we have 
\begin{equation}S_u(y, 4t)\subset S_u(0, \hat K t_0)\subset\subset\Omega.
 \label{Krole}
\end{equation}
 Then (\ref{bcf_small}) gives
 $$\|b\|_{L^p(S_u(y, 4t))}+ \|c^{-}\|_{L^n(S_u(y, 4t))} + \|f^{+}\|_{L^n(S_u(y, 4t))}\leq \e_4.$$
Thus, all the hypotheses of Lemma 
\ref{res_lem} are satisfied on $S_u(y, 4t)$. Applying this lemma, we conclude from (\ref{upper_den}) that $v>M^k$ on $S_u(y, t)$. This proves the Claim.\\
Using the Claim and Lemma \ref{inkspots}, we obtain
$$ |A_{k+1}| \leq (1-c\delta) |A_k|, $$
and therefore, by induction, $$|A_k| \leq (1-c_2\delta)^{k-1} (1-\delta) |S_u(0, t_0)| = C_2 M^{-\eps k}|S_u(0, t_0)|,$$ where $\e = -\log(1-c_2\delta) / \log M$ and 
$C_2=(1-c_2\delta)^{-1}(1-\delta)$.
This finishes the proof.
\end{proof}
\begin{rem}
We need a universal constant $\hat K$ in the above lemma to guarantee
the validity of (\ref{Krole}). 
If $u(x)=\frac{1}{2}|x|^2$ then $S_u(y, t)= B_{\sqrt{2t}}(y)$ and (\ref{Krole}) holds when $\hat K=4$. Under (\ref{pinch1}) only, sections of $u$ can have degenerate 
geometry, namely, they can be long and thin in different directions. 
\end{rem}

\section{Harnack inequality}
\label{Harnack_sec}
In this section, we prove Theorem \ref{CGthm}. It follows from the following theorem and a covering argument.
\begin{thm} [Harnack inequality for section with small height]
\label{CGthm2}Assume that (\ref{pinch1}) and (\ref{unilamU}) are satisfied in $\Omega$. 
Suppose that $v\geq 0$ is a $W^{2, n}_{loc}(\Omega)$ solution of (\ref{LMAeq}) in 
a section $S:=S_u (x_0, h)\subset\subset \Omega$ 
where 
$f\in L^n (S), c\in L^n(S)$ and $b\in L^p(S)$ 
with $p>\frac{n(1+\alpha_\ast)}{2\alpha_\ast}$. 
There exists a universal constant $\e_5(n,p,\lambda,\Lambda,\tl,\tL)>0$ with the following property.  
If $h\leq h_0$ where
$$h_0^{\frac{\alpha_\ast}{1+\alpha_\ast}-\frac{n}{2p}}\|b\|_{L^p(S)}\leq \e_5,~\text{and}~h_0^{1/2}\|c\|_{ L^n (S)}\leq \e_5,$$
then
\begin{equation}\sup_{S_{u}(x_0, h/8)} v\leq C(n, \lambda, \Lambda,\tl,\tL) \left(\inf_{S_{u}(x_0, h/8)} v+ h^{1/2}\|f\|_{L^n(S)}\right).
 \label{HI3}
\end{equation}
\end{thm}
Given Theorem \ref{CGthm2}, we can prove our main result as stated in Theorem \ref{CGthm}.
\begin{proof}[Proof of Theorem~\ref{CGthm}] Let $\e_5(n,p,\lambda,\Lambda,\tl,\tL)$ and $h_0$ be as in Theorem \ref{CGthm2}.
By Theorem \ref{pst}, there is a small, universal constant $\tau(n,\lambda,\Lambda)\in (0, 1/8)$ such that if
 $x\in\overline{S_u(x_0, h)}=:D$ then
 $$S_u(x,8\tau h)\subset S_u(x_0, 2h).$$
We first consider the case $h\geq h_0$.  
 By Lemma \ref{Vita_cov} (ii), we can select from $D\subset \bigcup_{x\in D} S_u (x,\tau h_0)$ a finite covering 
 \begin{equation}
  \label{finite_cov}
  D\subset \bigcup_{i=1}^m S_u (x_i,\tau h_0)
 \end{equation}
where the sections $\left\{S_u(x_i, \frac{\tau}{K} h_0)\right\}$ are mutually disjoint. Using the volume estimates of sections in Lemma \ref{vol-sec1} and
$$\bigcup_{i=1}^m S_u(x_i, \frac{\tau}{K} h_0)\subset S_u(x_0, 2h),$$
we deduce that
$$m c(n,\lambda,\Lambda) |h_0|^{n/2}\leq\sum_{i=1}^m |S_u(x_i, \frac{\tau}{K} h_0)|\leq |S_u(x_0, 2h)|\leq C(n,\lambda,\Lambda)|h|^{n/2}.$$
Hence
$$m\leq C(n,\lambda,\Lambda) \left(\frac{h}{h_0}\right)^{n/2}.$$
Applying Theorem \ref{CGthm2} to each section $S_u(x_i, 8\tau h_0)\subset S_u(x_i, 8\tau h)\subset S_u(x_0, 2h)\subset\subset\Omega,$
we get
\begin{equation*}
 \sup_{S_{u}(x_i, \tau h_0)} v^n\leq C(n, \lambda, \Lambda,\tl,\tL) \left(\inf_{S_{u}(x_i, \tau h_0)} v^n + h_0^{n/2}\|f\|^n_{L^n(S_u(x_0, 2h))}\right).
\end{equation*}
Therefore, by the volume estimate of sections in Lemma \ref{vol-sec1}, we find
\begin{equation*}
 \sup_{S_{u}(x_i, \tau h_0)} v^n\leq C(n, \lambda, \Lambda,\tl,\tL) \left(\inf_{S_{u}(x_i, \tau h_0)} v^n + |S_u(x_i, \frac{\tau}{K} h_0)|\|f\|^n_{L^n(S_u(x_0, 2h))}\right).
\end{equation*}
Combining this estimate with (\ref{finite_cov}), we discover
\begin{equation*}
 \sup_{S_{u}(x_0, h)} v^n\leq C(n, \lambda, \Lambda,\tl,\tL)^m \left(\inf_{S_{u}(x_0, h)} v^n + |S_u(x_0, 2h)|\|f\|^n_{L^n(S_u(x_0, 2h))}\right)
\end{equation*}
from which we deduce
$$\sup_{S_{u}(x_0, h)} v\leq C(n, \lambda, \Lambda,\tl,\tL)^{N(h, h_0)} \left(\inf_{S_{u}(x_0, h)} v + h^{1/2}\|f\|_{L^n(S_u(x_0, 2h))}\right)$$
as asserted. Here, we recall that \begin{equation*}
 N(h, h_0)=\max\left\{1, \left(\frac{h}{h_0}\right)^{n/2}\right\}.
\end{equation*}
It remains to consider the case $h\leq h_0$. The proof is similar to that above but we include it here.
By Lemma \ref{Vita_cov} (ii), we can select from $D\subset \bigcup_{x\in D} S_u (x,\tau h)$ a finite covering 
 \begin{equation}
  \label{finite_cov2}
  D\subset \bigcup_{i=1}^m S_u (x_i,\tau h)
 \end{equation}
where the sections $\left\{S_u(x_i, \frac{\tau}{K} h)\right\}$ are mutually disjoint. Using the volume estimates of sections in Lemma \ref{vol-sec1} and
$\displaystyle\bigcup_{i=1}^m S_u(x_i, \frac{\tau}{K} h)\subset S_u(x_0, 2h),$
we deduce that
$m\leq C(n,\lambda,\Lambda)$.
Applying Theorem \ref{CGthm2} to each section $S_u(x_i, 8\tau h)\subset S_u(x_0, 2h)\subset\subset\Omega,$
we get
\begin{equation*}
 \sup_{S_{u}(x_i, \tau h)} v\leq C(n, \lambda, \Lambda,\tl,\tL) \left(\inf_{S_{u}(x_i, \tau h)} v + h^{1/2}\|f\|_{L^n(S_u(x_0, 2h))}\right).
\end{equation*}
Combining this estimate with (\ref{finite_cov2}), we discover
$$\sup_{S_{u}(x_0, h)} v\leq C(n, \lambda, \Lambda,\tl,\tL)\left(\inf_{S_{u}(x_0, h)} v + h^{1/2}\|f\|_{L^n(S_u(x_0, 2h))}\right)$$
as asserted. The proof of Theorem \ref{CGthm} is complete.
\end{proof}

The rest of the section is devoted to the proof of Theorem \ref{CGthm2}.
\begin{proof}[Proof of Theorem~\ref{CGthm2}] 
Let $\delta\in (0, 1)$ and $M>1$ be the constants in Proposition \ref{decay_rem} and $\e\in (0, 1)$ be the constant in Theorem \ref{decay_thm}.
Let $\e_4$ and $c_\ast$ be as in Lemma \ref{res_lem}. We choose $\e_5$ so that $$C\e_5\leq \frac{\e_4}{16 M}$$
where $C$ is the universal constant appearing in (\ref{tbLp})- (\ref{tfLn}).
We rescale our equation, the domain, and functions as in Section \ref{AIP_sec}. Since $h\leq h_0$, by (\ref{tbLp}) and (\ref{tcLn}), the functions $\tilde b, \tilde c$
satisfy on the normalized $\tilde S= T^{-1}(S_u(x_0, h))= S_{\tilde u}(y, 4t_0)$ the bound
\begin{equation}\|\tilde b\|_{L^p(\tilde S)}+ \|\tilde c\|_{L^n(\tilde S)} \leq C\e_5\leq \frac{\e_4}{16 M}.
 \label{bcsmall}
\end{equation}
We need to show that
\begin{equation}\sup_{S_{\tilde u}(y, t_0/2)} \tilde v\leq C(n, \lambda, \Lambda,\tl,\tL) \left(\inf_{S_{\tilde u}(y, t_0/2)} \tilde v+ \|\tilde f\|_{L^n(\tilde S)}\right).
 \label{HI4}
\end{equation}
Without loss of generality, we can assume that $t_0=1$ and $y=0$. 
{\it We now drop the tildes in our argument.}
By changing coordinates and subtracting an affine function from $u$,
we can assume that $u\geq 0, u(0)=0$, $Du(0)=0$.
For simplicity, we denote $S_t= S_u(0, t)$. The theorem follows from the following Claim.\\
{\bf Claim.} If $\inf_{S_u(0, 1/2)} v \leq 1 ~\text{and}~\|f\|_{L^n(S)}\leq \frac{\e_4}{16 M}$ then for some universal constant $C$, we have $\sup_{S_u(0, 1/2)} v \leq C.$

Indeed, for each $\tau>0$, the function $$v^{\tau}= \frac{v}{\inf_{S_u(0, 1/2)} v  +\tau + 16M \|f\|_{L^n(S)}/\e_4 }$$
satisfies
$a^{ij} v^{\tau}_{ij} + b\cdot Dv^{\tau} + c v^{\tau}= f^{\tau}$
where
$$f^{\tau}=  \frac{f}{\inf_{S_u(0, 1/2)} v  +\tau + 16M \|f\|_{L^n(S)}/\e_4 }.$$
Thus,
$$\|b\|_{L^p(S)}+ \|c\|_{L^n(S)} \leq \frac{\e_4}{16M}, \|f^{\tau}\|_{L^n(S)}\leq \frac{\e_4}{16M}.$$
We apply the Claim to $v^{\tau}$ to obtain
$$\sup_{S_u(0,1/2)} v \leq C \left(\inf_{S_u(0, 1/2)} v  +\tau + 16M \|f\|_{L^n(S)}/\e_4\right).$$
Sending $\tau\rightarrow 0$, we get the conclusion of the theorem.

We now prove the Claim, following the line of argument in Imbert-Silvestre \cite{IS} in the case $u(x)=\frac{1}{2}|x|^2$.
Let $\beta >0$ be a universal constant to be determined later and let $h_t(x) = t(1-u(x))^{-\beta}$ be
defined in $S_u(0, 1)$. We consider the minimum value of $t$ such that
$h_t \geq v$ in $S_u(0, 1)$. It suffices to show that $t$ is bounded by a universal constant $C$, because we have then
$$\sup_{S_u(0, 1/2)} v\leq C \sup_{S_u(0, 1/2)} (1- u)^{-\beta}\leq 2^\beta C.$$
If $t \le 1$, we are
done. Hence, we further assume that $t \ge 1$. 

Since $t$ is chosen to be the minimum value such that $h_t \geq
v$, then there must exist some $x_0 \in S_u(0, 1)$ such that $h_t(x_0) =
v(x_0)$. Let $r = (1-u(x_0))/2$.  Let $H_0 := h_t(x_0) = t(2r)^{-\beta}
\ge 1$. By Theorem \ref{pst},
there is a small constant $\hat c$ and large constant $$p_1=(n+1)\mu^{-1}$$ such that 
\begin{equation}S_u(x_0, \hat K \hat c r^{p_1})\subset S_u(0, 1),
 \label{sectx0}
\end{equation}
 where $\hat K$ is the constant in Lemma 
\ref{KKlem}.

We bound $t$ by estimating the measure of the set $\{v \geq H_0/2\} \cap
S_u(x_0, \hat c r^{p_1})$ from above and below. The estimate from above can be done using Theorem
\ref{decay_thm}. First, recalling $S= \tilde S= S_u(0, 4)$, we find from (\ref{bcsmall}), $\|f\|_{L^n(S)}\leq \frac{\e_4}{16 M} $ and (\ref{sectx0})
that
$$\|b\|_{L^p(S_u(x_0, \hat K \hat c r^{p_1}))}+ \|c\|_{L^n(S_u(x_0, \hat K \hat c r^{p_1}))} + \|f\|_{L^n(S_u(x_0, \hat K \hat c r^{p_1}))}\leq \e_4.$$
Then, Theorem \ref{decay_thm} tells us that
\begin{equation} \label{up_H}  
|\{v>H_0/2\} \cap S_u(x_0, \hat c r^{p_1})|\leq
CH_0^{-\eps}|S_u(x_0, \hat c r^{p_1})|\leq C H_0^{-\eps}|S_u(0, 1)| \leq C t^{-\eps} (2r)^{\beta \eps}.
\end{equation}

To estimate the measure of $\{v \geq H_0/2\} \cap
S_u(x_0, \hat c r^{p_1})$ from below, we apply Lemma \ref{res_lem} to $C_1-C_2v$ on a small but definite fraction of this section.  Let $\rho$ be the small universal
constant and $\beta$ be a large universal constant such that
\begin{equation}\label{beta_choice}
 M \left( (1-\rho)^{-\beta} - 1 \right)  =
 \frac 12,~
\beta  \geq \frac{n(n+1)}{2\mu\e}.
\end{equation}
Consider the section $S_u(x_0, c_1 r^{p_1}) $ where $c_1\leq c_\ast$. We claim that $1- u(x)\geq 2r-2\rho r$ in this section if $c_1$
is universally small. Indeed, if $x\in S_u(x_0, c_1 r^{p_1}) $
then by Theorem \ref{sec-size}, we have $|x-x_0| \leq C (c_1 r^{p_1})^{\mu} \leq c_1^{\mu/2} \rho r$ for small, universal $c_1$. Hence, by the gradient estimate in Lemma \ref{slope-est}
\begin{eqnarray*}
1-u(x) = 2r + u(x_0)- u(x) \geq 2r -(\sup_{S_u(0, 1)}|\nabla u|) |x-x_0| \geq 2r -C(n,\lambda,\Lambda)  c_1^{\mu/2} \rho r\geq 2r-2\rho r.
\end{eqnarray*}
The maximum of $v$ in the section $S_u(x_0, c_1 r^{p_1})$ is at most the maximum
of $h_t$ which is not greater than $t(2r-2\rho r)^{-\beta} =
(1-\rho)^{-\beta} H_0$. 
Define the 
following functions for $x\in S_u(x_0, c_1 r^{p_1}) $

$$ w(x) = \frac{(1-\rho)^{-\beta} H_0
  -v(x)}{\left((1-\rho)^{-\beta} - 1
  \right) H_0},~\text{and }\bar f =\frac{c(1-\rho)^{-\beta} H_0-f}{\left((1-\rho)^{-\beta} - 1
  \right) H_0}. $$ Note that $w(x_0) = 1$, and $w$ is a non-negative solution of 
\begin{equation}a^{ij}w_{ij} + b \cdot Dw +  c w= \bar f~\text{in}~S_u(x_0, c_1 r^{p_1}). 
 \label{weq}
\end{equation}
Observe that, by (\ref{beta_choice}), (\ref{bcsmall}) and the assumption on $f$ in our Claim,
\begin{eqnarray*}
 \|\bar f\|_{L^n(S)}&\leq& \frac{c(1-\rho)^{-\beta}}{\left((1-\rho)^{-\beta} - 1
  \right) } \|c\|_{L^n(S)}+ \frac{1}{\left((1-\rho)^{-\beta} - 1
  \right) } \|f\|_{L^n(S)}\\
  &=& (2M+1) \|c\|_{L^n(S)} + 2M \|f\|_{L^n(S)} \leq \e_4/2.
\end{eqnarray*}
Therefore,
\begin{equation}\|b\|_{L^p(S_u(x_0, c_1 r^{p_1}))}+ \|c\|_{L^n(S_u(x_0, c_1 r^{p_1}))} + \|\bar f\|_{L^n(S_u(x_0, c_1 r^{p_1}))}\leq \e_4.
 \label{bcbarf}
\end{equation}
From (\ref{weq}) and (\ref{bcbarf}), we can use Lemma \ref{res_lem} to obtain
\begin{equation} |\{w \leq M\} \cap S_u(x_0, 1/4 c_1 r^{p_1})| \geq \delta |S_u(x_0, 1/4 c_1 r^{p_1})|. 
 \label{finalest}
\end{equation}
In terms of the original function $v$, this is an estimate of a subset of $S_u(x_0, 1/4 c_1 r^{p_1})$ where $v$ is larger than
\[ H_0 \left((1-\rho)^{-\beta} - M \left(
( 1-\rho)^{-\beta} - 1 \right) \right) \geq
\frac{H_0}2, \]
because of the choice of $\rho$ and $\beta$ in (\ref{beta_choice}). Thus, we obtain from (\ref{finalest}) the estimate
\[ |\{v \geq H_0/2\} \cap S_u(x_0, c_1 r^{p_1})| \geq \delta |S_u(x_0, c_1 r^{p_1})|.\]
Recall that $p_1= (n+1)\mu^{-1}$. In view of  \eqref{up_H}, and the volume estimate on sections in Lemma \ref{vol-sec1}, we find
$$C t^{-\eps} (2r)^{\beta \eps}\geq  \delta |S_u(x_0, c_1 r^{p_1})|\geq c_3(n,\lambda,\lambda) r^{np_1/2} = c_3(n,\lambda,\lambda) r^{\frac{n(n+1)}{2\mu}},$$
for some universally small $c_3$.
By the choice of $\beta$ in (\ref{beta_choice}), we find that $t$ is universally bounded. The proof of Theorem \ref{CGthm2} is now complete.
\end{proof}
{\bf Acknowledgments.} The author would like to thank the anonymous referee for the pertinent comments and the careful reading 
together with the numerous corrections of the original manuscript.


\begin{thebibliography}{xx}
\bibitem{AS} Armstrong, S. N.; Smart, C. K.
Regularity and stochastic homogenization of fully nonlinear equations without uniform ellipticity. 
{\it Ann. Probab.} {\bf 42} (2014), no. 6, 2558--2594. 
\bibitem{Cab} Cabr\'e, X. Nondivergent elliptic equations on manifolds with nonnegative curvature. {\it Comm. Pure Appl. Math.} {\bf 50} (1997), no. 7, 623--665.
\bibitem{C2} Caffarelli, L. A.  Interior $W^{2,p}$
estimates for solutions to the Monge-Amp\`ere equation.
{\it Ann. of Math.} {\bf 131} (1990), no. 1, 135--150.
\bibitem{CC} Caffarelli, L. A.; ~Cabr\'e, X. {\em Fully nonlinear elliptic
equations.} American Mathematical Society Colloquium Publications,
volume 43, 1995.
\bibitem{CG} Caffarelli, L.A.; Guti\'errez, C.E. Properties of the solutions of the linearized Monge-Amp\`ere equations. {\it Amer. J. Math.} {\bf 119} (1997), no. 2, 423--465.
\bibitem{DPFS} De Philippis, G.; Figalli, A.; Savin, O. A note on interior $W^{2, 1+\e}$ estimates for the Monge-Amp\`ere equation. {\it Math. Ann.} {\bf 357} (2013), 11--22. 
\bibitem{GT} Gilbarg, D.; Trudinger, N.S. {\em Elliptic partial differential equations of second order}, 
Reprint of the 1998 edition. Classics in Mathematics. Berlin: Springer, 2001.
\bibitem{G} Guti\'errez, C. E. {\em The Monge-Amp\`ere equation.}
  Birkha\"user, Boston, 2001.
  \bibitem{IS} Imbert, C.; Silvestre, L. Estimates on elliptic equations that hold only where the gradient is large.
   {\it J. Eur. Math. Soc.} {\bf 18} (2016), no. 6, 1321--1338.
\bibitem{John} John, F. Extremum problems with inequalities as subsidiary conditions. {\it Courant Anniversary Vol.}, pages 187-204, 1948.
  \bibitem{KS2} Krylov, N. V.; Safonov, M. V.
A property of the solutions of parabolic equations with measurable coefficients. (Russian)
{\it Izv. Akad. Nauk SSSR Ser. Mat.} {\bf 44} (1980), no. 1, 161--175, 239. 
\bibitem{Mal} Maldonado, D. Harnack's inequality for solutions to the linearized Monge-Amp\`ere operator with lower-order terms. 
{\it J. Differential Equations} {\bf 256} (2014), no. 6, 1987--2022.
\bibitem{Mooney} Mooney, C. Harnack inequality for degenerate and singular elliptic equations with unbounded drift. 
{\it J. Differential Equations} {\bf 258} (2015), no. 5, 1577--1591.
\bibitem{Safo} Safonov, M. V. Non-divergence elliptic equations of second order with unbounded drift. 
 {\it Nonlinear partial differential equations and related topics}, 211--232, Amer. Math. Soc. Transl. Ser. 2, 229, Amer. Math. Soc., Providence, RI, 2010. 
\bibitem{Sa} Savin, O. Small perturbation solutions for elliptic equations. {\it Comm. Partial Differential Equations} {\bf 32} (2007), no. 4-6, 557--578.
\bibitem{Sch} Schmidt, T. $W^{2, 1+\e}$ estimates for the Monge-Amp\`ere equation. {\it Adv. Math.} {\bf 240} (2013), 672--689. 
\bibitem{Tr} Trudinger, N. S. Local estimates for subsolutions and supersolutions of general second order elliptic quasilinear equations. 
{\it Invent. Math.} {\bf 61} (1980), no. 1, 67--79.
\bibitem{TW3} Trudinger, N. S.; Wang, X. J. The Monge-Amp\`{e}re equation and its 
geometric applications.  {\it Handbook of geometric analysis.} No. 1,  467--524, {\bf Adv. Lect. Math. (ALM), 7}, Int. Press, Somerville, MA, 2008.
\bibitem{W} Wang, Y. Small perturbation solutions for parabolic equations. {\it Indiana Univ. Math. J.} {\bf 62} (2013), no. 2, 671--697.
\end{thebibliography}
\end{document}